\documentclass[12pt,twoside]{amsart}
\usepackage[margin=3cm]{geometry}
\usepackage[colorlinks=false]{hyperref}
\usepackage[english]{babel}
\usepackage{graphicx,titling}
\usepackage{float}
\usepackage{amsmath,amsfonts,amssymb,amsthm}
\usepackage{lipsum}
\usepackage[T1]{fontenc}
\usepackage{fourier}
\usepackage{color}
\usepackage[latin1]{inputenc}
\usepackage{esint}
\usepackage{caption}
\usepackage{ccicons}

\makeatletter
\def\blfootnote{\xdef\@thefnmark{}\@footnotetext}
\makeatother

\newcommand\ccnote{
    \blfootnote{\copyright\,\, Max Engelstein and Dana Mendelson}
    \blfootnote{\ccLogo\, \ccAttribution\,\, Licensed under a \href{https://creativecommons.org/licenses/by/4.0/}{Creative Commons Attribution License (CC-BY)}.}
}

\usepackage[export]{adjustbox}
\numberwithin{equation}{section}
\usepackage{setspace}

\renewcommand{\leq}{\leqslant}

\renewcommand{\geq}{\geqslant}
\renewcommand{\mathbb}{\varmathbb}
\usepackage{fancyhdr}
\newtheorem{theorem}{Theorem}[section]
\newtheorem{lemma}[theorem]{Lemma}
\newtheorem{corollary}[theorem]{Corollary}
\newtheorem{proposition}[theorem]{Proposition}
\newtheorem{definition}[theorem]{Definition}
\newtheorem{remark}[theorem]{Remark}
\newcommand{\bR}{{\mathbb R}}

\newcommand{\bZ}{{\mathbb Z}}

\newcommand{\bS}{{\mathbb S}}
\newcommand{\bT}{{\mathbb T}}

\newcommand{\cN}{{\mathcal{N}}}

\newtheorem*{ntheorem}{Theorem}

\address{Max Engelstein, University of Minnesota-Twin Cities, Department of Mathematics}
\email{mengelst@umn.edu}
\medskip
\address{Dana Mendelson, University of Chicago, Department of Mathematics} 
\email{d.s.mendelson@gmail.com}

\begin{document}

\thispagestyle{empty}

\begin{minipage}{0.28\textwidth}
\begin{figure}[H]
\includegraphics[width=2.5cm,height=2.5cm,left]{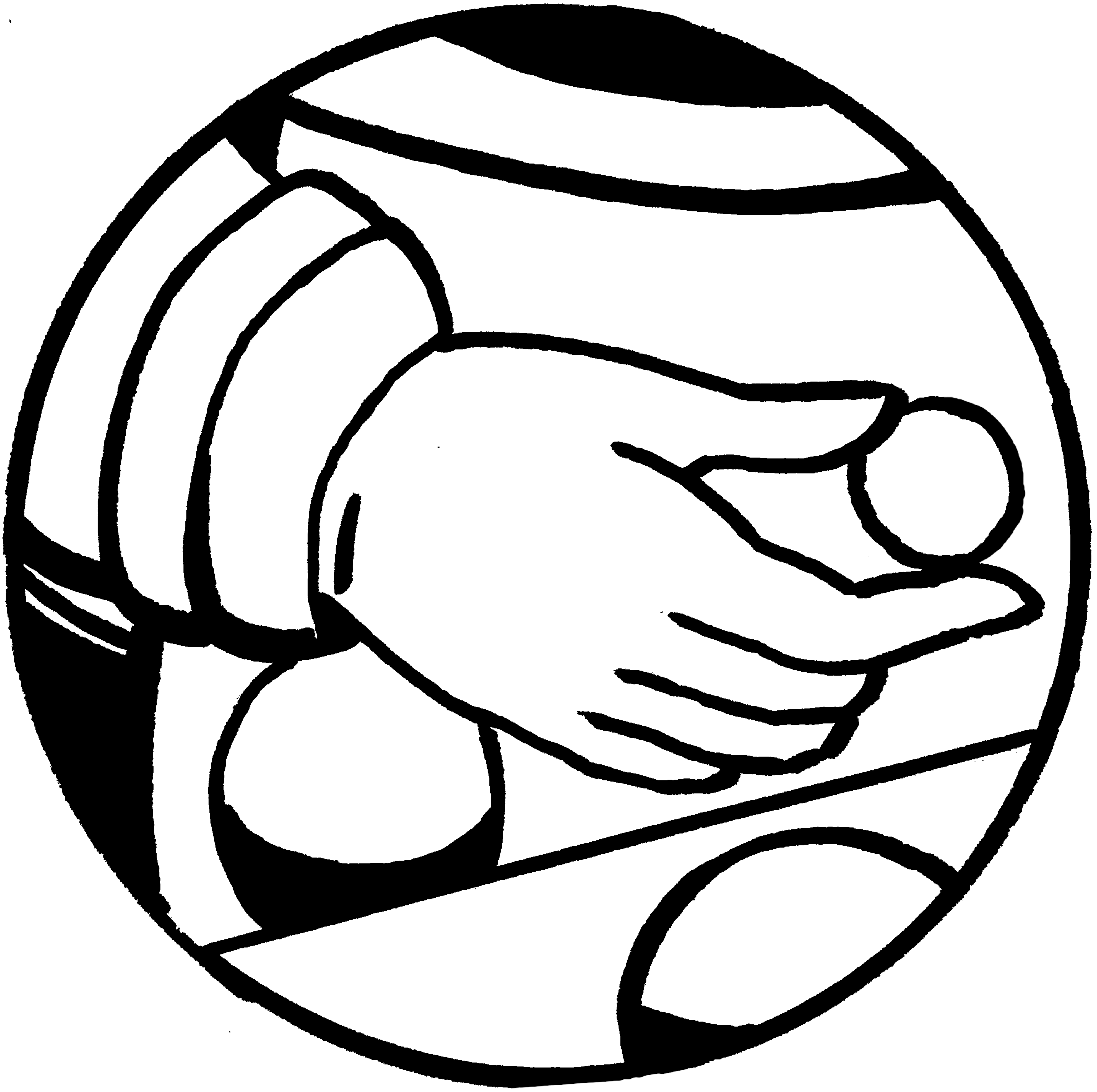}
\end{figure}
\end{minipage}
\begin{minipage}{0.7\textwidth} 
\begin{flushright}
Ars Inveniendi Analytica (2022), Paper No. 5, 30 pp.
\\
DOI 10.15781/kz11-np83
\\
ISSN: 2769-8505
\end{flushright}
\end{minipage}

\ccnote

\vspace{1cm}


\begin{center}
\begin{huge}
\textit{Non-uniqueness of bubbling for wave maps}
\end{huge}
\end{center}

\vspace{1cm}


\begin{minipage}[t]{.28\textwidth}
\begin{center}
{\large{\bf{Max Engelstein}}} \\
\vskip0.15cm
\footnotesize{University of Minnesota}
\end{center}
\end{minipage}
\hfill
\noindent
\begin{minipage}[t]{.28\textwidth}
\begin{center}
{\large{\bf{Dana Mendelson}}} \\
\vskip0.15cm
\footnotesize{University of Chicago}
\end{center}
\end{minipage}
\hfill

\vspace{1cm}


\begin{center}
\noindent \em{Communicated by Frank Merle}
\end{center}
\vspace{1cm}


\noindent \textbf{Abstract.} \textit{We consider wave maps from $\mathbb R^{2+1}$ to a $C^\infty$-smooth Riemannian manifold, $\mathcal N$. Such maps can exhibit energy concentration, and at points of concentration, it is known that the map (suitably rescaled and translated) converges weakly to a harmonic map, known as a bubble. We give an example of a wave map which exhibits a type of non-uniqueness of bubbling. In particular, we exhibit a continuum of different bubbles at the origin, each of which arise as the weak limit along a different sequence of times approaching the blow-up time.} 

\textit{This is the first known example of non-uniqueness of bubbling for dispersive equations. Our construction is inspired by the work of Peter Topping \cite{toppingwinding}, who demonstrated a similar phenomena can occur in the setting of harmonic map heat flow, and our mechanism of non-uniqueness is the same `winding' behavior exhibited in that work.}
\vskip0.3cm

\vspace{0.5cm}


\section{Introduction}\label{s:intro}
We consider wave maps from $(1+2)$-dimensional Minkowski space into a compact Riemannian manifold $(\cN, g)$, which are defined formally as critical points of the Lagrangian
\[
\mathcal{L}(U, \partial U) = \frac{1}{2} \int_{\bR^{1 + 2}} \eta^{\alpha \beta}\, \langle \partial_\alpha U , \partial_\beta U \rangle_g dt dx,
\]
where $\eta$ is the Minkowski metric and g is the Riemannian manifold on $\cN$.  In local coordinates on $\cN$, wave maps are solutions $u$ to the system
\begin{equation}\label{eq:wm}
 \left\{ \begin{aligned}
&\Box u^i = \Gamma_{k\ell}^i(u) \nabla u^k \cdot \nabla u^\ell, \\
&(u, \partial_t u ) \big|_{t=0} = (u_0, u_1),
\end{aligned} \right.
\end{equation}
where $\Gamma_{k\ell}^i$ are the Christoffel symbols for $\cN$.

There is a conserved energy associated to \eqref{eq:wm} given by
\begin{equation}\label{eq:energy}
\mathcal{E}(u, \partial_t u )(t) = \int_{\mathbb{R}^2} |\nabla  u |_g^2+|\partial_t u |_g^2 \equiv \mathcal{E}(u_0, u_1 ),
\end{equation}
whenever the right-hand side is finite. The coercivity of this energy implies that that the $\dot{H}^1$-norm of solutions remains bounded for all time. However, it is still possible for energy to concentrate producing ``solitons" or ``bubbles" in the suitably rescaled weak limit. Bubbling is a well-studied phenomena in nonlinear evolution equations, particularly in the parabolic setting, dating back to the work of Struwe \cite{Struwe85}, which built on the work of Sacks and Uhlenbeck \cite{SacksUhlenbeck}. We will not attempt to do justice to the vast literature here, particularly pertaining to parabolic flows, but we refer to \cite{delpino} and references therein for an overview on the history.  The purpose of this paper is to provide the first example, in the context of \eqref{eq:wm} {\it where this weak limit is non-unique},  see Theorem~\ref{t:main} below. In particular, we construct a solution where, after rescaling around a particular point in space-time,  different solitons are obtained by considering weak limits along different sequences of times. 

\medskip
To more precisely state our results, let us recall some background on wave maps.  We consider \eqref{eq:wm} for smooth, finite energy initial data belonging to a certain symmetry class and satisfying certain quantitative higher regularity bounds; we will make these assumptions precise below.  For such initial data, classical energy methods show that \eqref{eq:wm} admits a unique smooth solution for small times, see e.g. \cite{shatahstruwe}.  When the domain is $\mathbb R^{2+1}$, the energy \eqref{eq:energy} is scale invariant, thus solutions to \eqref{eq:wm} can exhibit energy concentration even if the initial data is smooth or highly symmetric, see, e.g. \cite{KST08, RR12, rodster}. This concentration can be ruled in certain cases under additional assumptions, say on the smallness of the initial energy, see \cite{taosmalldata, tatarurough, kriegerhyperbolic}, or by restricting the topology of the target, see, e.g. \cite{struweequivariant, kriegerschlagbook, taoarxiv}. 

\medskip
For general targets, Sterbenz and Tataru \cite{sterbenztataruwavemaps1} showed that inside any light cone a dichotomy holds: either energy does not concentrate too much, in which case solution can be smoothly extended to the whole space with quantitative control, or the energy concentrates and the solution ``bubbles off" a harmonic map. To state their result (in the infinite time setting), we follow the notation in \cite{sterbenztataruwavemaps1}, and fix
\[
C_{[t_0, t_1]} = \{t_0 \leq t \leq t_1, r \leq t\}
\]
for the truncated light cone and
\[
S_{t} = \{(x,t)\mid |x| \leq t \}
\]
for the time sections of the light cone. We set
\[
\mathcal{E}_{S_t} = \frac{1}{2} \int_{S_t}  |\nabla u |_g^2 + |\partial_t u|_g^2.
\]

\begin{theorem}[\protect{\cite[Theorem 1.5]{sterbenztataruwavemaps1}}] \label{thm:st_dich}
Let $u: C_{[1,\infty)} \to \mathcal{N}$ be a $C^\infty$ wave map such that
\[
\lim_{t \to \infty} \mathcal{E}_{S_t}[u] < \infty.
\]
 Then exactly one of the following possibilities must hold:
 \begin{itemize}
\item[(a)] There exists a sequence of points $(t_n, x_n) \in C_{[1,\infty)}$ and scales $\lambda_n$ with
\[
t_n \to \infty, \quad \limsup_{n \to \infty} \frac{|x_n|}{t_n} < 1, \quad \lim_{n \to \infty} \frac{\lambda_n}{t_n} = 0.
\]
so that the rescaled sequence of maps
\[
u^{(n)}(t,x) = u(t_n + \lambda_n t, x_n + \lambda_n x)
\]
converges strongly in $H^1_{loc}$ to a Lorentz transform of an entire (time-independent) harmonic map  
\[
u^{(\infty)} : \mathbb{R}^2 \to \mathcal{N}
\]
 of nontrivial energy:
\[
u^{(\infty)} : \mathbb{R}^2 \to \mathcal{N}, \quad 0 < \|u^{(\infty)} \|_{\dot H^1(\bR^2)} \leq \lim_{t \to \infty}  \mathcal{E}_{S_t}[u].
\]
\item[(b)] For each $\varepsilon > 0$, there exists $t_0 >  1$ and a wave map extension
\[
u : \mathbb{R}^2 \times [t_0, \infty) \to \mathcal{N}
\]
with bounded energy:
\[
\mathcal{E}(u) \leq (1 + \varepsilon^8) \lim_{t \to \infty}  \mathcal{E}_{S_t}[u],
\]
which satisfies
\[
\sup_{t \in [t_0, \infty)} \sup_{k \in \mathbb{Z}} ( \|P_k u(t) \|_{L^\infty_x} + 2^{-k} \|P_k \partial_t u(t) \|_{L^\infty_x}) \leq \varepsilon,
\]
where $P_k$ are the Littlewood-Paley projections to frequency $2^k$.
 \end{itemize}
An analogous dichotomy holds for finite time blowup, see \cite[Theorem 1.3]{sterbenztataruwavemaps1}. 
\end{theorem}

\begin{remark}
In the previous theorem and throughout, we use $\dot H^s(\bR^n)$ to denote the usual homogeneous Sobolev space, with norm defined as
\[
\|u\|_{\dot H^s(\bR^n)} = \|(-\Delta)^{s/2} u\|_{L^2(\bR^n)}.
\]
\end{remark}

\begin{remark}\label{r:scattering}
It was observed in \cite{sterbenztataruwavemaps1} that part (b) in Theorem \ref{thm:st_dich} implies that a certain controlling norm for $u$ is finite, and in \cite{sterbenztataruwavemaps2} it was proved that this implies that $u$ converges to a linear wave after applying a suitable gauge transformation. We note that this is called scattering in the terminology of \cite{sterbenztataruwavemaps1}, but is different from the usual use of that term in nonlinear dispersive equations, see \cite[Remark 1]{lawire_oh_cmp} for some discussion. We adopt this terminology for brevity when referring to the behavior in part (b) of this theorem.

Instead of working with the modified definition of scattering in \cite{sterbenztataruwavemaps2}, we instead adopt a more direct approach to demonstrate that our flow exhibits energy concentration via degree considerations. We believe this approach illuminates certain features of scattering in the current setting. This argument is inspired by arguments in \cite{CTZ93, StruweLecNotes, ckls2,  lawire_oh_cmp}, see Section \ref{s:flow} for more details. 
\end{remark}

The first case in Theorem \ref{thm:st_dich} describes ``bubbling'' (a.k.a.~energy concentration) to non-constant harmonic maps. Note that the bubbling phenomenon described in Theorem \ref{thm:st_dich} leaves open the question of whether the convergence along a discrete sequence of times and scales can be strengthened to convergence as $t\rightarrow \infty$ after rescaling, translating and applying Lorentz transformations by some continuous functions, $\lambda(t)$, $x(t)$ and $\gamma(t)$.

To reduce the number of parameters, and to provide an example of non-uniqueness in what we speculate is the simplest possible setting, we will impose an additional symmetry on our solution, which will allow us to ignore translations and Lorentz transformations: 

\begin{definition}\label{d:quasiequivariant}
We say that a wave map $u(x,t): \bR^{1+2} \rightarrow \mathcal N$ is quasi-equivariant if there exists a smooth one-parameter family of isometries, $\Phi_s \in \mathrm{Isom}(\mathcal N)$. with $\Phi_0 = \mathrm{Id}_{\mathcal N}$, such that (using polar coordinates) one has $u(r, \theta + s, t) = \Phi_s \circ u(r, \theta, t)$ for all $r > 0, s, \theta \in \mathbb S^1$ and $t\in [0, T_{\max})$. 
\end{definition}

\begin{remark}\label{r:quasiequiv}
Many of Topping's constructions for harmonic map heat flow (see, e.g. \cite{toppingwinding, toppingrigidity}) exhibit quasi-equivariant symmetries. While, Topping does not give this symmetry a name, we introduce the term ``quasi-equivariance" to stress its relation to the well-studied notion of equivariant wave maps. In the language of Definition \ref{d:quasiequivariant}, if $\mathcal N$ is a surface of rotation and $\Phi_s$ is a rotation of $\mathcal N$ by $s$ radians around the same axis, then $u$ is equivariant.
\end{remark}

We can now define uniqueness of bubbling for quasi-equivariant wave maps:

\begin{definition}\label{d:solitonresolution}
We say that a quasi-equivariant wave map, $u:\mathbb R^{1+2}\rightarrow \mathcal N$, has a unique bubble at a point $0\in \mathbb R^2$ and time $T_{max} > 0$ if there exists a continuous function, $\lambda(t): [0, T_{\max})\rightarrow (0, + \infty)$, with 
\begin{equation}
\lambda(t) = \begin{cases} o(T_{max} - t) & T_{max} < \infty\\
 o(t) &T_{max} =\infty \end{cases}
\end{equation}
such that $u(t, \lambda(t)x) \rightarrow \omega(x)$ in $C_{\mathrm{loc}}^0(\mathbb R^2\backslash \{0\}; \mathcal N)$ for some non-trivial harmonic map $\omega: \mathbb R^{2}\rightarrow \mathcal N$. 
\end{definition}

In contrast to part $(a)$ of Sterbenz-Tataru's Theorem \ref{thm:st_dich}, which yields convergence to a harmonic map along a sequence of times, the previous definition describes convergence (after application of the relevant symmetries in the quasi-equivariant class) to a unique harmonic map along \emph{all times}. In general, it is difficult to prove that uniqueness in the sense of Definition \ref{d:solitonresolution} fails. However, our Theorem \ref{t:nonuniqueness} shows that if the convergence is ``winding" (see Definition \ref{d:winding}), then the bubble is not unique. 

Using this approach, our main theorem gives the first example of a bubbling solution to a non-linear wave equation for which uniqueness in the sense of Definition \ref{d:solitonresolution} is known not to hold. 

\begin{theorem}[Main Theorem]\label{t:main}
There exists a compact smooth Riemannian manifold $(\mathcal{N}, g)$ given by
\[
\mathcal{N} = \mathbb{T}^2 \times_f \mathbb{S}^2
\]
for a certain $C^\infty(\bT^2)$ warping function, $f$\footnote{Recall that if $(M, g_1), (N,g_2)$ are Riemannian manifolds and $f: M\rightarrow \mathbb R$, is positive, then the warped product manifold, $M\times_f N$, is a Riemannian manifold consisting of the product space $M\times N$ endowed with the Riemannian metric $g := g_1 + f^2g_2$.}, and $C^\infty$-smooth, finite energy, quasi-equivariant initial data $(u_0, u_1)$, which satisfy
\[
\|(u_0, u_1) \|_{\dot H^3 \times \dot H^2} < \infty, \qquad \mathcal E(u_0, u_1) < \mathcal E_{\mathrm{quasi}}(\mathcal N) + \varepsilon_1,
\]
such that the corresponding solution $(u, u_t)$ to \eqref{eq:wm} develops a bubble as $t \to T_{\max}$ which fails to be unique in the sense of Definition \ref{d:solitonresolution}. Above, $\varepsilon_1 >0$ is a constant which depends only on $\mathcal N$, and $\mathcal E_{\mathrm{quasi}}(\mathcal N)$ denotes the smallest energy of a non-trivial quasi-equivariant harmonic map, $\omega: \mathbb S^2 \rightarrow \mathcal N$.
\end{theorem}

The energy assumption above is reminiscent of work which considers wave maps with energies slightly above the ``ground state" (that is, the lowest energy of a non-trivial harmonic map into the target manifold). We do not show that $\mathcal E_{\mathrm{quasi}}(\mathcal N)$ is the energy of the ground state, but it plays that role within our considered symmetry class and we will sometimes abuse terminology and refer to it as the energy of the ground state (see Lemma \ref{lem:bubbling_loc} and the remark after). 

\begin{remark}\label{r:remarksonmain}
A few clarifying remarks on Theorem \ref{t:main}:
\begin{itemize}
\item  Bubbling for radially symmetric wave maps into general targets has been ruled out under a number of different weak assumptions \cite{Struweradial, Nahas, CKL18}. Thus it is perhaps unreasonable to expect that Theorem \ref{t:main} could hold under stronger symmetry assumptions on the initial data. 
\item The fact that $f\in C^\infty$, as opposed to analytic, allows $f(p)-f(q)$ to vanish to arbitrarily high order as $p\rightarrow q$. Much previous work on wave maps has assumed non-degeneracy conditions (e.g. conditions (A1)-(A3) in Section 3 of \cite{jiakenig}) which rule out this behavior. 
\item In addition to $C^\infty$ smoothness and finite energy, we require some additional quantitative regularity of the initial data. While $\dot H^{1+ \varepsilon} \times \dot H^{\varepsilon}$ regularity (for any $\varepsilon > 0$) would suffice, we assume $\dot H^3 \times \dot H^2$-regularity of the initial data since the proof of local wellposedness proceeds by classical energy methods and is simpler than the techniques required to get the almost critical result. This additional regularity is used in Lemma \ref{l:imageofflow}.
\item The proof of Theorem \ref{t:main} actually establishes the existence of an infinite family of quasi-equivariant initial data $(u_0, u_1)$ for which the corresponding solution has non-unique bubbling, as opposed to a specific construction of the initial data. 
\end{itemize}
\end{remark}

\begin{remark}\label{rem:rate_discussion}
We do not, currently have more precise control on the rates $\{r_n\}$ for the bubbling in Theorem \ref{t:main} than what is provided in \cite{sterbenztataruwavemaps1}. In the case of equivariant wave maps into $\mathbb{S}^2$, it is remarked in \cite{ckls2} that various possibilities exist, and we refer to \cite{jendrej17, Pillai} for the construction of an infinite-time blow-up in this setting. In \cite{toppingwinding}, Topping establishes a lower bound on the rate of blow-up in the context of harmonic map heat flow. A key ingredient in Topping's proof of this lower bound is a quantitative ``neck estimate'', see \cite[Lemma 4.4]{toppingwinding}, due to Qing and Tian \cite{qingtian}, see also Lemma 2.9 in \cite{toppingannals}. While neck estimates have been recently made available for wave maps into spheres in the work of Grinis \cite{Grinis}, those estimates are non-quantitative, and thus it remains unclear how to use them to gain control on the rate of bubbling for wave maps, even if we were to adapt them to the case of general targets.  
\end{remark}

\subsection{Soliton Resolution Conjecture and Prior Work}\label{ss:priorwork}
One motivation for establishing Theorem \ref{t:main} stems from the large amount of recent activity in establishing the  {\it soliton resolution conjecture} for nonlinear wave and wave map equations.  The soliton resolution conjecture posits that solutions for a broad class of nonlinear dispersive equations should decompose asymptotically as a sum of ``bubbles" and radiation.  In the setting of (quasi-)equivariant wave maps, $\psi: \mathbb R^{2+1} \rightarrow \mathcal M$, the conjecture states that if $\psi$ develops at least one bubble at $t = \infty$, then there exists a collection of finite energy harmonic maps, $\{Q_j\}_{j=1}^J: \mathbb R^2 \rightarrow \mathcal M$, continuous scaling parameters, $0 < \lambda_1(t) \ll \lambda_2(t) \ll \ldots \ll \lambda_J(t) \ll t$, and a finite energy linear wave, $\phi_L$, such that 
\begin{equation}\label{e:solitonres}
\psi(-, t) = \phi_L(-, t) + \sum_{j=1}^J  Q_j\left(\frac{-}{\lambda_j(t)}\right) + \varepsilon(t),
\end{equation}
where $\varepsilon(t) \rightarrow 0$ in the appropriate function space as $t\rightarrow \infty$ (see the introduction of \cite{cote, Grinis} for a more detailed discussion of the soliton resolution conjecture for wave maps into the sphere).  In particular, \eqref{e:solitonres} implies uniqueness in the sense of Definition \ref{d:solitonresolution} for each of the bubbles, $Q_j$. We note that such a description goes beyond Theorem \ref{thm:st_dich} which, together with uniqueness in the sense of Definition \ref{d:solitonresolution}, only describes the dynamics of energy concentration at one scale. We refer to the examples of \cite{jendrejconstruction, jendrejlawrie} which demonstrate that such multi-scale concentration is in fact possible. 

Much progress has been made on the soliton resolution conjecture a variety of settings. Thus far, the full conjecture has been proved in the work of Duyckaerts, Kenig and Merle \cite{dkm3d} and \cite{dkmfullsoliton} for the radial, focusing energy-critical nonlinear wave equation in odd space dimensions, and wave maps into $\mathbb S^2$ under various symmetry and energy assumptions, see, for instance, \cite{ckls1, ckls2, djkm2, jendrejlawrie} and references therein\footnote{Since this article was first posted on the arXiv, full soliton resolution for ($k$-)equivariant wave maps into $\mathbb S^2$ has been established in \cite{DKMM} and \cite{JL2}}. Nonetheless, in many cases the conjecture in its full strength remains open. Often, one can show that an asymptotic decomposition in the vein of \eqref{e:solitonres} holds along a sequence of times $t_n \rightarrow T_{\max}$, see \cite{cote, jiakenig, djkm}. In these instances, the difficulty becomes proving that the decomposition is independent of the sequence $\{t_n\}$.  Our main theorem, Theorem \ref{t:main}, demonstrates that in some settings, it may be impossible to move beyond the soliton resolution along a sequence of times to the full conjecture.

\medskip
We hope that our theorem may shed some light on the specific difficulties in proving continuous time soliton resolution. Indeed, for quasi-equivariant non-linear wave equations there are two main ways in which the decomposition could depend on the sequence of times. The first is that the bubbles can ``switch places", i.e.  $Q_j$ may develop at a smaller scale than $Q_k$ along one sequence of times but at a larger scale along another sequence of times.
The second is that the bubbles, $Q_j$, themselves could depend on the sequence of times considered. In both cases, a careful understanding of the potential bubbles, i.e. non-trivial harmonic maps at specified energy levels, and possible interactions between bubbles separated in scale or in space has appeared to be a crucial ingredient in proofs which ultimately establish uniqueness for such problems  (see, e.g. \cite{jendrejlawrie, dkmfullsoliton} respectively). 

A full understanding of the possible bubbles and their interactions is often achieved by imposing symmetry assumptions; for example, there is one (up to rescaling) radially symmetric stationary solution to the non-linear wave equation considered in \cite{dkm3d}. In other situations, energy constraints can rule out multi-soliton configurations, see, for instance, \cite{ckls1, ckls2} which establishes continuous-time soliton resolution for equivariant wave maps into $\mathbb S^2$ with energy less than three times the energy of the lowest energy harmonic map into the sphere, and \cite{djkm2} for the result without the equivariance assumption when the energy is restricted to just above the energy of the ground state. In contrast to these cases, the target manifold in Theorem \ref{t:main} admits a continuum of quasi-equivariant harmonic maps $\omega: \mathbb R^2 \rightarrow \mathcal N$ at the lowest admissible non-zero energy level, and the richness of this family plays an essential role in the proof of non-uniqueness for Theorem~\ref{t:main}. 

Finally, we note one further aspect of our setting which contributes to non-uniqueness, specifically compared to the setting considered in \cite{djkm2}. While in both cases wave maps with energy just above the ground state are considered, in  \cite{djkm2} the authors exploit the fact that for wave maps into spheres, the energy is coercive near the traveling waves, which traps the wave map in increasingly small neighborhoods of the traveling wave, yielding uniqueness. In contrast, while the second component of our target manifold is the sphere, the first component has an infinite length geodesic which, using standard coordinates  $\mathbb{T}^2 = (w, z)$, wraps around the torus infinitely many times as it approaches the circle $\{w = 0\}$. This winding behavior allows the first coordinate of the wave map to exit a small neighborhood of a  soliton infinitely many times, even though the behavior of the second component of the map will be well controlled (i.e. the second component may be almost constant).

\subsection{Comparison with non-uniqueness in elliptic and parabolic problems}\label{ss:lojasiewicz}

For the harmonic map heat flow (the parabolic version of \eqref{eq:wm}), the analogue of the soliton resolution conjecture is also open, even when the target is $\mathbb S^2$, despite being an object of intense study.  For a small sample of this work, see \cite{toppingannals, rupflin, waldron} on questions of uniqueness and \cite{delPinoWei, toppingthesis} for constructions of solutions that blowup at multiple different points simultaneously or form non-trivial bubble trees at a single point. 

In \cite{toppingwinding}, Topping showed that the analogue of soliton resolution cannot hold for harmonic map heat flow without additional assumptions on the target manifold; our construction is heavily inspired by that paper. In particular, our target manifold, including the warping function, is essentially equivalent to Topping's (more on this in Section \ref{s:target}). We also use the same mechanism as Topping to ensure non-uniqueness, what he calls ``winding". Note that we record only the definition of winding as it applies in the quasi-equivariant setting, and a more general definition would take into account translational and Lorentz symmetries.

\begin{definition}[\protect{\cite[Definition 1.12]{toppingwinding}}]\label{d:winding}
A quasi-equivariant wave map $u: \mathbb R^{1+2} \rightarrow \mathcal M$ develops a winding bubble at time $T_{\max}$ and the origin $ 0 \in \mathbb R^2$ if there exists sequences $\{\lambda_n\}$, and $\{t_n\}$, satisfying
\begin{equation}
 t_n \uparrow T_{\max}, 
 \qquad \lambda_n = \begin{cases} o(T_{max} - t_n) & T_{max} < \infty \\
 o(t_n) & T_{max} = \infty \end{cases}
\end{equation}
such that $u(t_n, \lambda_nx) \rightarrow \omega(x)$ in $C^0_{\mathrm{loc}}(\bR^2\backslash \{x_0\}; \mathcal M)$, where $\omega$ is a non-constant harmonic map, the lifts $\hat{u}(t_n,\lambda_nx)$ have no convergent subsequence in $C^0_{\mathrm{loc}}(\bR^2\backslash \{x_0\}; \widehat{\mathcal M})$ (here and throughout, $\widehat{\mathcal M}$ is the universal cover of $\mathcal M$). 
\end{definition}

Despite the similarities in set-up, the execution of our proof differs substantially from Topping's. This is due to fundamental differences between the parabolic and Hamiltonian settings, even at the level of ODEs, which we had to overcome. To elucidate these issues, we elaborate here on two examples.

For analytic functions, $f$, bounded gradient flows, $\dot{x}(t) = -\nabla f(x(t))$, have a unique limit as $t\rightarrow \infty$, due to the \L ojasiewicz inequalities \cite{lojasiewicz}. It is a beautiful observation of L. Simon \cite{simonLoj} that this fact about ODEs can be applied to study the long time behavior of the gradient flows of many elliptic functionals which arise naturally in geometry. These ``\L ojasiewicz-Simon" inequalities have found subsequent use in a huge range of geometric and variational problems. For example, to prove uniqueness of tangent objects for variational problems (e.g. minimal surfaces, \cite{simonLoj} and mean curvature flow \cite{comi}) and to show the uniqueness of long time limits of geometric flows (e.g. Yamabe flows \cite{brendleyamabe}, harmonic map heat flow \cite{Feehan1}). However, this phenomena does not hold in the Hamiltonian setting. Indeed, if $f(x_1,x_2) = \frac{1}{2}(x_1^2 +x_2^2)$ then the equation $\ddot{x}(t) = -\nabla f(x(t))$ becomes $\ddot{x}(t) = -x(t)$. One solution to this ODE is the bounded flow $x(t) = (\cos(t), \sin(t))$, which clearly does not have a unique limit as $x(t) \rightarrow \infty$. 

When $f$ is not analytic, but is $C^\infty$, then the classic ``goat tracks" example \begin{equation}\label{e:goattracks} f(r,\theta) = \begin{cases} 1&\quad r \leq 1,\\
1 + e^{-\frac{1}{r-1}}\left(\sin(\frac{1}{r-1} + \theta) + 2\right)&\quad r > 1,
\end{cases}\end{equation} generates a gradient flow $\dot{x}(t) = -\nabla f(x(t))$ which is bounded but doesn't have a unique limit as $t\rightarrow \infty$, in fact, every point on the circle $r = 1$ is a limit point. This example is at the heart of Topping's construction \cite{toppingwinding} (see also the examples of non-uniqueness for singularities of harmonic maps, \cite{whitenonunique}, and the long term behavior of harmonic map heat flow, \cite{toppingrigidity}). In contrast, the corresponding flow given by $\ddot{x}(t) = -\nabla f(x(t))$ cannot exhibit the same asymptotic behavior as the gradient flow does.  The Hamiltonian flow stays bounded as long as $|x(0)|$ is close to one and $\dot{x}(0)$ is small enough, however, by working in polar coordinates, one can see that if $|x(t)| \rightarrow 1$ as $t\rightarrow \infty$ it must also be the case that $|\dot{x}(t)| \rightarrow 0$ as $t\rightarrow \infty$. This would violate the conservation of the energy, $|\dot{x}|^2(t) + f(x(t))$ (provided $\dot{x}(0) \neq 0$).

As these examples show, one needs caution when using long-term behavior of non-linear parabolic flows to provide insight into the Hamiltonian setting. On a practical level, while energy conservation provides some control for Hamiltonian flows, it is not a substitute for the maximum principle and energy dissipation, which hold in the parabolic setting. For example, as in \cite{toppingwinding} we show that the image of the flow is contained in a geodesic in the target manifold, see Lemma~\ref{l:imageofflow}. However, we are faced with the additional difficulty of showing that the flow cannot leave this geodesic before the blow-up time. We note that this issue is not present in the parabolic setting, where stationary solutions act as a barrier to constrain the flow.  Additionally, energy dissipation allows Topping to determine that his flow blows up in finite time via topological considerations and Lemaire's theorem. On the other hand, we must leave open the possibility that the winding singularity may occur at either finite or infinite time (we speculate either situation can occur).

Nonetheless, we believe that our proof exhibits, morally, a phenomenon exploited in work of  Grinis \cite{Grinis} (and observed earlier in work of \cite{CTZ93, sterbenztataruwavemaps1, sterbenztataruwavemaps2}): that nonlinear wave equations start to exhibit elliptic behavior in the (strict) interior of a light-cone in which energy is concentrating. 

\subsection{Structure of the Paper}
Here we briefly outline the structure of the paper. In Section \ref{s:wavemapprelim} we record some preliminaries about wave maps and harmonic maps into general Riemannian manifolds.

In Section \ref{s:target} we construct the target manifold $\mathcal{N}$ from Theorem \ref{t:main}. This follows much as in \cite[Section 3]{toppingwinding}, with additional complications caused by the fact that we are unable to use the maximum principle to constrain the image of the flow. We overcome this difficulty by working with a compact target manifold and carefully defining our metric twist globally on $\bT^2$. 

In Section \ref{s:flow} we establish some preliminary results on the Hamiltonian flow. In subsection \ref{ss:scattering} we also rule out ``scattering" to a solution of the underlying linear wave equation. In \cite{toppingwinding} the analogous parabolic phenomena, relaxation to a stationary solution, can be quickly ruled out using topology. The wave map setting requires a more involved estimate of the energy flux through the wall of the light cone. 

In Section \ref{s:harmonicmaps} we study the harmonic maps into $\mathcal{N}$ which can arise as bubbles in the flow. Here the analysis is complicated by the fact that the metric twist had to be defined globally on $\bT^2$. In particular, we use energy arguments to rule out bubbles which wrap ``the wrong way" around the $\bT^2$ component, see Lemma \ref{lem:bubbling_loc}. 

 Finally, in Section \ref{s:winding} we establish properties of the bubbling, in particular winding, proving Theorem \ref{t:main}.

\subsection{Acknowledgements}
This work was done while ME was visiting for the AY 2019-2020 the University of Chicago; he thanks the department and especially Carlos Kenig for their hospitality. He also thanks Carlos Kenig for suggesting he investigate the phenomena of bubbling for harmonic map heat flow. DM learned about Topping's work on non-uniqueness for the harmonic map heat flow bubbling, and the open question of non-uniqueness for wave maps bubbling, in a topics course given by Carlos Kenig at the University of Chicago. Both ME and DM benefited from several discussions with Andrew Lawrie, including discussions about Topping's work on harmonic map heat flow and about scattering for wave maps. The authors thank both Carlos Kenig and Andrew Lawrie for helpful comments on an earlier version of this manuscript. The authors also thank several anonymous referees for their helpful comments and close reading. 

\section{Preliminaries for Wave Maps and Harmonic Maps}\label{s:wavemapprelim}
In this section we collect some basic facts about the regularity of (quasi-equivariant) wave maps and harmonic maps. As mentioned in the introduction, we will consider wave maps
\[
u: \mathbb{R}^{1+2} \to \mathbb{T}^2 \times_f \mathbb{S}^2 \equiv \mathcal N 
\]
for a certain $C^\infty(\bT^2)$ warping function, $f$. Using polar coordinates in the domain of $u$, and coordinates $(x,y, \alpha, \theta)$ in the target,\footnote{$\alpha$ is the polar angle and runs from $0$ to $\pi$, $\theta$ is the azimuth angle which runs from $0$ to $2\pi$. We define the $x,y$ coordinates on $\bT^2$ through a change of variables in \eqref{e:cotangentchange}.} we suppose that the initial condition has the form
\begin{equation}\label{e:initconditions}
u_0(r,\theta) = (0, y_0, \alpha_0(r), \theta),\qquad u_1(r,\theta) = (0, y_1(r), \alpha_1(r), 0).
\end{equation}

We will rely on the following local existence and persistence of regularity result, which follows via energy estimates and holds for wave maps in general. Note too that the optimal local theory is known, see for instance \cite{KlS97}.

\begin{proposition}[Classical local existence and persistence of regularity]\label{p:existence}
Let $s_0 > 1$ and let $(u_0, u_1) \in \dot H^{s_0 + 1} \times \dot H^{s_0}$. There exists a $T_{max} \equiv T_{max}(u) > 0$ such that for every $T < T_{max}$, there exists a unique solution
\[
u:[0,T] \times \mathbb{R}^{1+2} \to \mathcal{N}
\]
 of \eqref{eq:wm} such that
\begin{align}\label{eq:sobolev_bds}
\sup_{0 < t \leq T} \|u(t) \|_{\dot H^{s_0 + 1}} < \infty.
\end{align}
Moreover, if $(u_0, u_1) \in H^{s + 1} \times H^{s}$ for any $s > s_0$, then
\[
\sup_{0 < t \leq T} \| u(t) \|_{ \dot H^{s + 1}} < \infty.
\]
\end{proposition}

The uniqueness conclusion of Proposition \ref{p:existence} implies that the wave map retains rotational and quasi-equivarient symmetry. Therefore, the solution will have the form
\begin{align}\label{equ:sol_ansatz}
u(t, r,\theta) = (X(t,r), Y(t,r) , \alpha(t,r), \theta).
\end{align}

In the coordinates of \eqref{equ:sol_ansatz}, the energy has the form
\begin{equation}\label{e:energy} 
\mathcal{E}(t) = 2\pi \int_0^\infty \left(\frac{1}{2}|\nabla (X,Y)|^2 + \frac{1}{2}|\partial_t (X,Y)|^2 + f(X,Y)e(\alpha)\right)rdr, 
\end{equation} 
where $e(\alpha)$ is the ``spherical" part of the energy
\[
e(\alpha) := \frac{1}{2} \left[ \left( \frac{\partial \alpha}{\partial r}\right)^2 +\left( \frac{\partial \alpha}{\partial t}\right)^2+ \frac{\sin^2 \alpha}{ r^2} \right].
\]

If the initial conditions \eqref{e:initconditions} have finite energy, $\mathcal{E}(u_0,u_1) < \infty$, then $\mathcal{E}(t) \equiv \mathcal{E}(u_0, u_1)$ for all $t < T_{\max}$. Furthermore, if $u$ bubbles to $\omega$ at $T_{\max}$ in the sense of Theorem \ref{thm:st_dich} or the analogous finite time result, then as stated in case (a) of that theorem, $\int |D\omega|^2 \leq \mathcal{E}(u_0, u_1)$.

In order to satisfy 
\[
\int_0^\infty f(X,Y) e(\alpha) rdr < \infty
\]
for all $t$ it must be the case that $\alpha(t, 0) = m \pi$ and $\alpha(t, \infty) = n \pi$ for all $t \in I_{max}(u)$ for some integers $n, m$. That the integers must be constant for all $t$ follows from the continuity in time of the flow. We will take $m=0$ in the sequel, and we define the \emph{degree} of the wave map to be the integer $n$.  

\subsection{Regularity for Weakly Harmonic Maps from $\bR^2$}\label{ss:weakharmonic}

Let $(\mathcal M, g)$ be a closed smooth Riemannian manifold. Throughout we will assume that $\mathcal M$ is smoothly embedded into $\mathbb R^n$.

Weakly harmonic maps $\omega: \bR^2 \rightarrow \mathcal M \subset \mathbb R^n$ are critical points of the energy $E(u) \equiv \int_{\bR^2} |Du|^2_g$ under perturbations of the form $u_\varepsilon(x) = \pi (u + \varepsilon \varphi(x))$ where $\pi$ is the nearest point projection of $\mathbb R^n$ onto $\mathcal M$, which is well defined and smooth in a small neighborhood of $\mathcal M$, and $\varphi \in C^1_c(\bR^2; \bR^n)$. 

Equivalently, these maps (weakly) satisfy the semi-linear PDE: $$-\Delta \omega + A(\omega)(\nabla \omega, \nabla \omega) = 0$$ where $A(\omega)$ is the second fundamental form of $\mathcal M$. In general, weakly harmonic maps need not be continuous (see, e.g. \cite{riviere}), but when the domain is two-dimensional all weakly harmonic maps are $C^\infty$ by the fundamental result of \cite{helein}. As such, we will refer to weakly harmonic maps from $\mathbb R^2$ as simply harmonic maps. Finally, if $\omega: \bR^2 \rightarrow \mathcal M$ is harmonic, then (composing with a stereographic projection) we get a harmonic map $\widetilde{\omega}: \mathbb S^2\backslash\{\infty\} \rightarrow \mathcal M$, which we can smoothly extend to all of $\mathbb S^2$ by the work of Sachs-Uhlenbeck \cite{SacksUhlenbeck}. We will often abuse notation and identify the harmonic maps $\omega, \widetilde{\omega}$. 

\medskip
Our first theorem quantifies the regularity of two-dimensional harmonic maps (the precise statement for non-minimizing maps on $\bR^2$ is hard to track down. However, one can argue as in \cite{schoenuhlenbeck} or consider the stationary case of the parabolic regularity proven in \cite{struweflow}):

\begin{theorem}\label{t:hmregularity}
Let $\omega: \mathbb R^2 \rightarrow \mathcal M$ be a harmonic map. There exists an $\bar{\varepsilon} > 0$ small, and depending on $\mathcal M$, such that if \begin{equation}\label{e:smallenergy} \int_{B_R(x_0)} |D\omega|^2 \leq \bar{\varepsilon},\end{equation} then, \begin{equation}\label{e:interiorgradientestimate} R^2\sup_{B_{R/2}(x_0)}|D\omega|^2 \leq C\int_{B_R(x_0)} |D\omega|^2 dx,\end{equation} where $C > 0$ depends on $\mathcal N$ but not on $x_0, R$ or $\omega$.  Equivalently, given the small energy condition \eqref{e:smallenergy}
\begin{equation}\label{e:interioroscillationestimate} \sup_{x,y\in B_{R/2}(x_0)} d_{\mathcal M}(\omega(x),\omega(y)) \leq C \left(\int_{B_R(x_0)} |D\omega|^2 dx\right)^{1/2},
\end{equation} 
where $d_{\mathcal{M}}(p,q)$ is the geodesic distance between $p,q \in \mathcal{M}$ and $C > 0$ is as above.
\end{theorem}

There are two standard corollaries of Theorem \ref{t:hmregularity} which will be important to us. The first states that there is a least energy non-trivial harmonic map into any target. 

\begin{corollary}\label{c:leastenergy}
For any weakly harmonic map $\omega: \bR^2 \rightarrow \mathcal M$, one of following two hold: 
\begin{itemize}
\item $\mathcal E(\omega) \equiv \int_{\bR^2}|D\omega|^2 = 0$ and $\omega$ is almost everywhere equal to a constant $p \in \mathcal M$. 
\item $\mathcal E(\omega) \geq \bar{\varepsilon} > 0$, where $\bar{\varepsilon}$ is the constant from Theorem \ref{t:hmregularity}
\end{itemize}
\end{corollary}

The second corollary states that the shortest path between two points in the image of a harmonic map cannot have infinite length: 

\begin{lemma}\label{l:finitepaths}
Let $\omega: \bR^2 \rightarrow \mathcal M$ be a weakly harmonic map. There cannot be two points $p_1, p_2 \in \mathrm{Im}\,\omega$ such that the shortest path between the two points in $\mathcal N$ has infinite length. 
\end{lemma}

\begin{proof}
Let $x,y \in \bR^2$ and let $L_{x,y}$ be the line segment connecting the two. Let $E$ be the energy of the harmonic map. Since $\omega\in C^\infty(\bR^2)$,  for all $z \in L_{x,y}$ there exists an $r_z > 0$ such that 
\[
\int_{B_{r_z}(z)}|D\omega|^2 < \bar{\varepsilon}.
\]
By compactness there are finitely many $z_i \in L_{x,y}$ such that 
\[
L_{x,y}\subset \bigcup_{i=1}^M B_{r_{z_i}}(z_i).
\]
Using the oscillation estimate \eqref{e:interioroscillationestimate}, we have that $d_{\mathcal M}(\omega(x), \omega(y)) \leq CM\sqrt{\bar{\varepsilon}} < \infty$. 
\end{proof}

We end this section with some elementary facts about harmonic maps $\omega: \mathbb S^2 \rightarrow \mathbb S^2$. 

\begin{lemma}\label{lem:hmspheres}
Let $\omega$ be a harmonic map $\mathbb S^2 \rightarrow \mathbb S^2$.  Let $\mathcal E_{\mathbb S^2}$ denote the lowest energy level of a non-trivial harmonic map between spheres, which is guaranteed to exist by Corollary \ref{c:leastenergy}. There is a unique (up to a conformal transformation of $\mathbb{S}^2$) equivariant harmonic map $\omega(r,\theta) = (\alpha(r), \theta)$ such that $\mathcal E(\omega) = \mathcal E_{\mathbb S^2}$.  Furthermore, if $\omega$ is equivariant and $\mathcal E(\omega) > \mathcal E_{\mathbb S^2}$ then it must be that $\mathcal E(\omega) \geq 2 \mathcal E_{\mathbb S^2}$. 
\end{lemma}

\section{Construction of the target}\label{s:target}

We now construct the target manifold $\cN$ using a modification of Topping's construction in \cite{toppingwinding}. Recall that we will construct $\cN$ as a twisted product of the torus $\mathbb T^2$ with $\mathbb S^2$, i.e.
\[
\cN = \mathbb T^2 \times_f \mathbb S^2,
\]
where $f \in C^\infty(\bT^2)$ denotes the warping function.  We introduce coordinates $(w,z)$ on $\bT^2$ where $w,z\in [0,1] \cong \mathbb S^1$. From here on, we will refer to $\mathcal M$ when our theorems apply to arbitrary target manifolds and $\mathcal N$ when the target is the manifold constructed in this section. 

\medskip

We define the curve
\[
\gamma(s)= \left(\frac{1}{\pi}\cot^{-1}(s),\, s \mod 1\right), \quad s \in  (-\infty, \infty) ,
\]
where inverse cotangent is defined so that $\cot^{-1}(s): (-\infty, \infty) \rightarrow (0,\pi)$. Observe that $\cot^{-1}$ is $C^\infty$ with uniform control on all derivatives in any $[-K, K] \subset (-\infty, \infty)$. We want to define a metric on $\bT^2$ such that $\gamma$ is a geodesic. We claim that 
\[
h(w,z) = \begin{pmatrix} \pi^2 & \pi \sin(\pi w)^2\\\pi\sin(\pi w)^2 & 1+\sin(\pi w)^4\end{pmatrix}
\]
gives such a metric. To see that this is the case, and to simplify our analysis of the wave maps equation, it will be useful to make the following coordinate change: 

\begin{equation}\label{e:cotangentchange}
\Phi: (w,z) \mapsto (x,y)= (\cot(\pi w)-z,z),\qquad D\Phi = \begin{pmatrix} \frac{-\pi}{\sin^2(\pi w)} & -1\\0 & 1\end{pmatrix}
\end{equation}
We note that $D\Phi$ is well defined and invertible away from $w = 0 \equiv 1$, and we will show in Lemma \ref{l:imageofflow} that when $t < T_{\max}$ the flow stays away from this curve, hence the change of variables remains valid up until the first blow-up time. We further note that $\Phi$ respects the symmetry of $\bT^2$ given by $w\cong w+1$ and that the image of $\Phi$ will have the symmetry $(x,y)\cong (x+1, y-1)$, and thus will be a cylinder. We shouldn't expect the image of $\Phi$ to be a topological torus as $\Phi$ is not well defined at $w= 0, 1$.  

In the $(x,y)$ coordinates 
\[
\gamma(s) = (0, s), \quad s\in (-\infty, \infty)
\]
and the pushforward of the metric, $h$, is given by 
\[
\textbf{h}(x,y) = \begin{pmatrix} \frac{1}{(1+(x+y)^2)^2} & 0\\0 & 1\end{pmatrix}.
\]
 It is now straightforward to compute that $\gamma$ is a geodesic. As a check, we observe that $(x,y) \cong (x+1, y-1)$ is an isometry of $\mathrm{Im}\, \Phi$ equipped with the metric $\textbf{h}$.

\medskip
We now turn to the construction of the warping function $f(w,z)$. Here we face a technical difficulty not present in \cite{toppingwinding} which is our inability to constrain the flow using the maximum principle. Thus we need to carefully define the warping function on all of $\mathcal N$ and not just in a neighborhood of $w = 0$. We do this by means of a smooth cutoff function, the existence of which is guaranteed by the following lemma. 

\begin{lemma}
There exists a $C^\infty$ bump function $\chi: \bT^2 \rightarrow [0,1]$ with the following properties: 
\begin{equation}\label{e:chi}
\begin{cases} \chi \equiv 1 & w\in [0, 1/4]\cup [3/4, 1] \\
\chi \equiv 0 & w \in [7/16, 9/16]\\
\nabla \chi|_{\gamma} \parallel \dot{\gamma} & w\in (1/4, 3/4).
\end{cases}
\end{equation}
\end{lemma}

\begin{proof}

We provide only a sketch of this proof. To ease notation, let $X = \{(w,z) \mid w\in (1/4, 3/4)\} \subset \bT^2$. To define the cut-off function $\chi$, we smoothly connect a cut-off defined in a neighborhood of the geodesic $\gamma$, with another depending only on $w$ outside this neighborhood. The latter construction is simpler, so we only describe the former. Let $\gamma_\delta = \{p \in \bT^2\mid \mathrm{dist}(p, \gamma) < \delta\}$. Inside of $X$, $\gamma$ is parameterized by a $C^\infty$ function with uniform bounds on all the derivatives, therefore, there exists a $\delta_0 > 0$ such that we can smoothly parameterize $\gamma_{\delta_0}$ by $(t, n)$ where the $t$ coordinates are parallel to $\gamma$ and the $n$ coordinates are normal to $\gamma$. In $X \cap \gamma_{\delta_0/2}$, we define $\chi$ to depend only on $t$ (working in $t, n$ coordinates) and to $C^\infty$ interpolate between $0$ and $1$ so that the first two conditions of \eqref{e:chi} hold.
\end{proof}

With $\chi$ defined we can now define the metric twist $\tilde{f}$ in a neighborhood of $w = 0 \cong 1$:
\[
 \tilde{f}(w,z) \equiv e^{-2\pi\cot(\pi w)}\left(\sin 2\pi \bigl(\cot(\pi w)-z-1/8\bigr)+\sqrt{2}\right) + 1,
 \]
and then globally define the twist by
\begin{equation}\label{e:defoff}f(w,z) = \begin{cases}
1& \:\: w = 0 \\
\tilde{f}(w,z)& \:\: 0 < w \leq 1/4 \\
\chi(w,z)\tilde{f}(w,z)+ (1-\chi(w,z))M& \:\: 1/4 < w \leq 1/2 \\
\chi(w,z)\tilde{f}(1-w,1-z) + (1-\chi(w,z))M &\:\: 1/2 < w \leq 3/4\\
\tilde{f}(1-w, 1-z)&\:\: 3/2 < w < 1
\end{cases}\end{equation} 
where $M > 1$ is a constant to be chosen later. Note that $f$ is $C^\infty$ away from $w= 0,1$ (as it is the sum of products of $C^\infty$ functions). At $w = 0$ we observe that $\tilde{f} - 1$ vanishes to infinite order and similarly $\tilde{f}(1-w, 1-z)$ at $w = 1$. Thus $f\in C^\infty(\bT^2)$, and furthermore, $f$ is invariant under the isometries of the space, i.e. $(w,z) \cong (w+n, z+m)$ for $(n,m)\in \mathbb Z\times \mathbb Z$.

\medskip
In $(x,y)$ coordinates, we have
\[
\tilde{f}(x,y) = e^{-2\pi(x+y)}\left(\sin 2\pi\bigl(x-1/8\bigr)+\sqrt{2}\right) + 1.
\]

It is a straightforward calculation to see the following properties of $f$:
\begin{itemize}
\item[(i)] $\partial_x \tilde{f}(0,y) = 0$ 
\item[(ii)]  $\partial_y \tilde{f}(0, y) < 0$. 
\end{itemize}
Combined with what we know about $\chi$, this implies that $\nabla f|_\gamma$ is parallel to $\gamma$ in all of $\bT^2\backslash \{w= 0\}$, allowing for the fact that $\nabla f$ is zero away from the support of $\chi$.

\medskip
We can define $f$ globally in $(x,y)$ coordinates by

\begin{equation}\label{e:defoffiny}f(x,y) = \begin{cases}
\tilde{f}(-x, -y)&\:\: x+y < -\cot^{-1}(\pi/4)\\
\chi(x,y)\tilde{f}(-x,-y) + (1-\chi(x,y))M &\:\:  x+y \in [-\cot^{-1}(\pi/4) , 0) \\
\chi(x,y)\tilde{f}(x,y)+ (1-\chi(x,y))M& \:\: x+y  \in [0,  \cot^{-1}(\pi/4)) \\
\tilde{f}(x,y)& \:\: x+y  > \cot^{-1}(\pi/4) 
\end{cases}\end{equation} 
Furthermore, by picking $M > 2\sup_{\bT^2}\tilde{f}$  we can guarantee that 
\[
\mathrm{sgn}(y) \partial_y f(0,y) \leq 0\textup{ with equality if and only if }f \equiv M
\]
 (i.e. $y \notin \mathrm{supp} \chi$). This can be seen through a chain rule computation and the fact that  $\partial_y \tilde{f}(0, y) < 0$.

\medskip
We end this section by summarizing the properties of $f$ and $\mathcal N$ which are important to us. 
\begin{lemma}[Properties of the target Manifold]\label{l:propertiesofmanifold}
The function $f$, manifold $\mathcal N = \bT^2\times_f \bS^2$ and curve $\gamma$, described above, have the following properties:
\begin{enumerate}
\item \label{itmfat0} $f \in C^\infty(\bT^2)$ and $f \geq 1$ always with $f = 1$ iff $w = 0$.  
\item \label{gammageodesic} The curve $\gamma(s): (-\infty, \infty) \rightarrow \bT^2$ is a geodesic with the following properties: 
\begin{enumerate}
\item \label{itminfinitelength} For any $s\in \mathbb R$, $\ell(\gamma((-\infty,s)) = \infty = \ell(\gamma((s, \infty)))$.
\item \label{itmderivative} $\mathrm{sgn}(s)\frac{d}{ds}f(\gamma(s)) > 0$ except in a neighborhood of $0$, in which $f \equiv M \gg 1$
\item \label{itmfoutside1} $\{w = 0\} \subset \overline{\gamma((-\infty,\infty))}$ but $\{w = 0\} \cap \gamma(-\infty, \infty) = \emptyset$.
\item \label{itmparallelgradient} For any $s \in \bR$, $\nabla f(\gamma(s)) \parallel \gamma'(s)$. 
\end{enumerate}
\end{enumerate}
\end{lemma}

Let us quickly comment on some of these conditions:

\begin{remark}\label{r:analyticity} 
It is not so important that $f, \mathcal N$ satisfy the conditions \ref{itmfat0}, \ref{itmfoutside1} precisely. The arguments here work for any $f$ which is globally bounded away from zero and which achieves its minimum on a topological circle that is in the closure of (but does not intersect) a geodesic $\gamma$. These facts will help us control the image of possible bubbles; cf. Lemma \ref{lem:bubbling_loc}. 

 Second, conditions \ref{itminfinitelength},\ref{itmderivative} \ref{itmparallelgradient} imply that the gradient flow generated by $f$ starting at a point along the curve $\gamma$ will be bounded but will not have a unique limit as $t\rightarrow \infty$. Rather, each point in $\{w = 0\}$ will be an accumulation point of the flow. As we mentioned in Section \ref{ss:lojasiewicz}, this property would not possible if $f$ were analytic. 
\end{remark}

\section{Analysis of the Hamiltonian Flow}\label{s:flow}
We now turn to the setting of wave maps into $\mathcal{N}$. Our first result of this section establishes that for wave maps with initial conditions of the form \eqref{e:initconditions}, the flow stays inside $\mathrm{Im} \gamma \times_f \mathbb S^2$ for all $t < T_{max}(u)$. Throughout this section, we will denote by $P_1$ and $P_2$ the projection of $\cN$ onto its two-dimensional components: 
\[
P_1 : \cN \to \bT^2, \quad P_2 : \cN \to \bS^2.
\]

\begin{lemma}\label{l:imageofflow}
Let $u: \mathbb R^{2+1} \rightarrow \cN$ be a wave map with initial conditions \eqref{e:initconditions}, which also satisfy $\|(u_0, u_1)\|_{\dot H^3\times \dot H^2} < \infty$. Then for all $t <  T_{max}(u)$, using the notation of \eqref{equ:sol_ansatz}, $X(t,\cdot) \equiv 0$, and
\begin{align}\label{equ:y_bds}
- \infty < Y(t,\cdot) < \infty.
\end{align}
\end{lemma}

\begin{proof}
We prove that as long as the initial conditions lay in  $\gamma \times T\gamma$ we have $X(t,r) \equiv 0$. This follows from the fact that $\gamma \times \mathbb S^2$ lies totally geodesically within $\cN$ (cf. \cite{toppingwinding} Section 3), and then the classical fact that wave maps with initial conditions in totally geodesic submanifolds stay in that submanifold. 

Let us briefly sketch how this works: $\gamma$ is a geodesic in $\bT^2$ which implies that the image of $(\Delta_{\bT^2})|_\gamma$ is contained in $T\gamma \subset T\bT^2$. The first component of the wave map satisfies the equation 
\[
\partial^2_{tt} P_1\circ u = -\Delta_{\bT^2} P_1\circ u - \nabla f(P_1\circ u) e(\alpha).
\]
Since $\nabla f(P_1\circ u)$ and $\Delta_{\bT^2}P_1\circ u$ lie tangent to $\gamma$, the flow stays in $\gamma$ for the whole time of existence.

Finally, to obtain boundedness of $Y$, we note that by assumption $(0,y_0) \in \gamma$, so, to fix notation, suppose that $s_0 \in \mathbb{R}$ is such that $\gamma(s_0) = (0,y_0)$. Note that 
\[
\ell\bigl(\gamma\bigl((s_0, \infty)\bigr)\bigr) = \infty = \ell\bigl(\gamma\bigl((-\infty, s_0)\bigr)\bigr), \qquad  \forall s_0 \in (-\infty, \infty),
\]
so we will conclude by establishing that the image of a finite energy wave map with $\dot H^{3} \times \dot H^2$ bounds cannot contain an infinite length path, which will establish that $|Y(\cdot, t)| < \infty$ for all $t < T_{\max}$ proving \eqref{equ:y_bds}. This is where we rely on the persistence of regularity result from Proposition \ref{p:existence}. We note again that $\dot H^{1+} \times \dot H^{0+}$ bounds would suffice for this argument.

For any $r_0 > 0$ and $r_1, r_2 > r_0$, letting $b = (0, Y(t, r_2))$ and $a = (0, Y(t, r_1))$ and abusing notation so that $(a,b)$ refers to the portion of $\gamma$ connecting these two points, we have
\[
\ell(a,b)  \leq \int_{r_1}^{r_2} |DY | dr \leq \mathcal{E}(u) \frac{1}{r_0},
\]
Hence for any given time, the only point in the domain at which $Y(t, r)$ can be infinite is $r = 0$. However, for $p > 2$  and for any $r_1 \leq 1$, letting $c =Y(t, r_1)$ and $d = Y(t, 0)$, we further have that 
\begin{align*}
\ell(c,d) \leq  \left( \int_{0}^{r_1} |DY |^{p} r dr \right)^{1/p} & \leq \left( \int_{0}^{\infty} |DY |^{2} r dr \right)^{1-\theta} \left( \int_{0}^{\infty} |D^3Y |^{2} r dr \right)^{\theta}\\
& \leq C\bigl(t, \|(u_0, u_1)\|_{\dot H^3 \times \dot H^2}, \mathcal{E}(u_0, u_1)\bigr),
\end{align*}
where the first inequality is an application of H\"older and the second inequality holds for some $0< \theta \equiv \frac{1-\frac{2}{p}}{2} < 1$ by Gagliardo-Nirenberg. Hence $Y(t, 0)$ cannot pass through $\{\pm \infty\}$.
\end{proof}

\subsection{Energy Concentration at $T_{\max} = \infty$}\label{ss:scattering}
In this subsection we assume that $T_{\max} = +\infty$. We want to show that scattering (i.e. outcome (b) in Theorem \ref{thm:st_dich}) cannot occur unless the wave map is degree zero. 

 We begin by rewriting the system of equations for $Y$ and $\alpha$, which, in light of Lemma \ref{l:imageofflow}, is given by 
\begin{equation}\label{equ:wm_sys}
\begin{split}
\frac{\partial^2 Y}{\partial t^2} &= \frac{\partial^2 Y}{\partial r^2} + \frac{1}{r} \frac{\partial Y}{\partial r} - \frac{\partial f}{\partial y}(0, Y) e(\alpha), \,\,\,\,\qquad\qquad\qquad\qquad\qquad (Y, Y_t)\big|_{t=0} = (y_0, y_1), \\
\frac{\partial^2 \alpha}{\partial t^2} &= \frac{\partial^2 \alpha}{\partial r^2} + \frac{1}{r} \frac{\partial \alpha}{\partial r} -\frac{\sin(2\alpha)}{2 r^2} + \frac{1}{f(0,Y)} \frac{\partial f}{\partial y}(0, Y) \frac{\partial Y}{\partial r} \frac{\partial \alpha}{\partial r} , \qquad (\alpha, \alpha_t)\big|_{t=0} = (\alpha_0, \alpha_1).
\end{split}
\end{equation}

We plan on showing that energy concentrates inside the light cone $|x| < t$ as $t\rightarrow \infty$ (i.e. that outcome (a) of Theorem \ref{thm:st_dich} holds). We start by observing that the norms of derivatives of quasi-equivariant functions have radial symmetry:

\begin{proposition}\label{p:radialderivative}
Let $u: \mathbb R^{1+2} \rightarrow (\mathcal M,g)$ be a quasi-equivariant wave map, then $|\nabla_x u|_g, |\nabla_t u|_g$ and $\left\langle \nabla_x u, \nabla_t u\right\rangle_g$ are all radially symmetric functions. 
\end{proposition}

\begin{proof}
Recall that $u(r, \theta+s,t) = \Phi_s \circ u(r, \theta, t)$ for $\Phi_s \in \mathrm{Isom}(\mathcal M)$. In particular this implies that $D\Phi_s$ satisfies $$\left\langle D\Phi_s v, D\Phi_s w\right\rangle_{g(\Phi_s(p))} = \left\langle v, w\right\rangle_{g(p)}, \qquad \forall p \in \mathcal M, v,w \in T_p \mathcal M.$$ We can then compute,
\begin{align*} 
\left\langle \nabla_x u(r,\theta+s, t), \nabla_t u(r, \theta+s, t)\right\rangle_{g} =& \left\langle D\Phi_s\nabla_x u(r, \theta, t), D\Phi_s \nabla_t u(r, \theta, t)\right\rangle_g \\
=& \left\langle \nabla_x u(r, \theta, t), \nabla_t u(r, \theta, t)\right\rangle_g.
\end{align*}
The same argument applies for $|\nabla_x u|_g, |\nabla_t u|_g$. 
\end{proof}

We also record a standard H\"older regularity estimate for quasi-equivariant wave maps. Note that we abuse notation and use $|\cdot|$ to denote the distance within the manifold.

\begin{lemma}\label{lem:holder}
Let $u:\mathbb R^{1+2}\rightarrow \mathcal M$ be a quasi-equivariant finite energy wave map into a smooth manifolds $\mathcal M$ and $\{\Phi_s\}_{s\in \mathbb S^1}$ the associated smoothly parameterized one parameter family of isometries (see Definition \ref{d:quasiequivariant}). Then, for any $r_0 > 0$ there exists
$$
C \equiv C(\mathcal{E}(u_0, u_1), r_0, \sup_{s \in \bS^1}\|\partial_s \Phi_s\|_{C^\infty(T\mathcal M)}) > 0
$$
such that for any $t \in \mathbb{R}, \theta_1, \theta_2 \in \bS^1$ and any $r,s \in \mathbb{R}$ with $r_0 < r,s$ we have
\[
|u(t, r, \theta_1) - u(t, s, \theta_2)| \leq C \bigl( |r-s|^{1/2} + |\theta_1-\theta_2| \bigr).
\]
\end{lemma}

\begin{proof}
First, fix $r > 0$. For any $\phi, \theta \in \bS^1, t\in \bR$ we have 
\begin{align}\label{holder1}
|u(t, r,\theta) - u(t, r, \phi)| = |u(t,r, \theta) - \mathrm{\Phi}_{\phi-\theta}(u(t, r, \theta))| < \|\mathrm{Id}-\mathrm{\Phi}_{\phi-\theta}\|_{C^\infty(\mathcal M)} < C|\phi-\theta|.
\end{align}
 Note, a similar argument shows that for $r,s, t > 0$ $|u(t,r,\theta) - u(t,s,\theta)|$ is independent of $\theta$.

Fix $\theta\in \bS^1$, $s > r>  r_0 > 0$. Then, by H\"older's inequality,
\begin{align}\label{holder2}
|u(t,r,\theta) - u(t,s,\theta)|^2 &= \frac{1}{2\pi}\int_{0}^{2\pi}\left( \int_{r}^s \partial_r u(t, r',\theta) dr' \right)^2d\theta \\
&\leq \left(\frac{1}{2\pi}\int_0^{2\pi} \int_{r}^s |\partial_r u(t, r',\theta)|^2 r' dr'd\theta \right) \frac{|r-s|}{r_0} .
\end{align}
Putting \eqref{holder1} and \eqref{holder2} together, the conclusion follows. 
\end{proof}

For $0 < T < \infty$ and define, for any $A \geq 0$ \begin{equation}\label{d:flux}
\mathrm{Flux}(u, T, A) := \int_{r = T-A} |\nabla_x u(r, \theta, T) + \nabla_t u(r,\theta,T)|_g^2 d\sigma(\theta). 
\end{equation}

As suggested by its name, the flux measures the energy entering the (translated inwards by $A$) light cone at time $T$. We can see this in the following energy identity:

\begin{equation}\label{e:fluxidentity}
\int_{|x| < T_1-A} |\nabla_{x,t}u(x,T_1)|_g^2 dx - \int_{|x| < T_2-A}|\nabla_{x,t} u(x,T_2)|_g^2 dx = \int_{T_2}^{T_1} \mathrm{Flux}(u, t, A)dt.
\end{equation}

It follows from \eqref{e:fluxidentity} that  
\[
T\mapsto \int_{|x| < T-A} |\nabla_{x,t}u(x,T)|_g^2 dx
\] 
is monotone increasing for any $A \geq 0$. Since the integral is monotone increasing in $T$ and bounded by $E(u_0, u_1)$, it follows that 
\[
\lim_{T\rightarrow \infty} \int_{r < T-A} |\nabla_{x,t}u(x,T)|_g^2 dx
\]
exists and, therefore, 
\begin{equation}\label{e:fluxlimit}
\lim_{T\uparrow \infty} \mathrm{Flux}(u,T, A) = 0.
\end{equation}

Our  next proposition shows that the energy in a linear neighborhood of the boundary of the light cone still goes to zero in infinite time; this observation was first made in the setting of radial wave maps, c.f. \cite{CTZ93},  see also Proposition 2.1 in \cite{ckls2} and Lemma 4.1 in \cite{StruweLecNotes}.

\begin{proposition}\label{p:channelsgotozero}
Let $u$ be finite energy quasi-equivariant wave map, with $T_{\max} = +\infty$. Then for any $\lambda \in (0,1)$ we have \begin{equation}\label{e:channelszero}
 \int_{\lambda T < |x| < T-A} |\nabla_x u(x,T)|_g^2 + |\nabla_t u(x,T)|^2_g dx \to 0 \qquad \textup{as $T, A \to \infty$ for  $A \leq (1-\lambda)T$}.
\end{equation}
\end{proposition}

\begin{proof} 
We follow the notation from \cite{StruweLecNotes}, see also \cite{shatahstruwe}. Let 
\[
e := \frac{1}{2}\left(|\nabla_x u|_g^2 + |\nabla_t u|_g^2\right), \:\: m := \left\langle \nabla_x u, \nabla_t u\right\rangle, \:\: L := \frac{1}{2}\left(|\nabla_x u|_g^2 - |\nabla_t u|_g^2\right).
\]

From these we have the algebraic relations, \begin{equation}\label{e:algrel}\begin{aligned} \partial_t(re) - \partial_r(rm) =& 0\\
\partial_t(rm) - \partial_r(re) =& L.
\end{aligned}
\end{equation}

Note that these identities follow from the wave map equation \eqref{eq:wm} and Proposition \ref{p:radialderivative}, which, in this context, implies that $\partial_re = \partial_xe$ and similarly with $m$. 

It will be convenient to reparameterize $(t,r)$ space by the coordinates $\xi = t-r$ and $\eta = t+r$. We also introduce the quantities $\mathcal A^2 = r(e+m)$ and $\mathcal B^2 = r(e-m)$. Then $\partial_\xi \mathcal A^2 = L = -\partial_\eta\mathcal B^2$. From this equation and the algebraic observation that $8r^2(e^2 -m^2) \geq L^2$, we see that \begin{equation}\label{e:aandb}\begin{aligned}|\partial_\xi \mathcal A^2| \leq& \frac{C}{r}\mathcal B\\
|\partial_\eta \mathcal B^2|\leq& \frac{C}{r}\mathcal A.
\end{aligned}\end{equation}

Let $Q$ denote the quadrilateral in $(\xi, \eta)$ space with vertices 
\[
((1-\lambda)T,(1+\lambda)T), \quad (A, 2T-A), \quad (A, 2s-(1-\lambda)T), \quad((1-\lambda)T, 2s-(1-\lambda)T)
\]
where $s \gg T > 0$, see Figure \ref{f:quad} (the ordering above is the order of the vertices in the Figure, starting from the lower left and moving counter-clockwise).
\begin{figure}
\begin{center}\includegraphics[width=.35\textwidth]{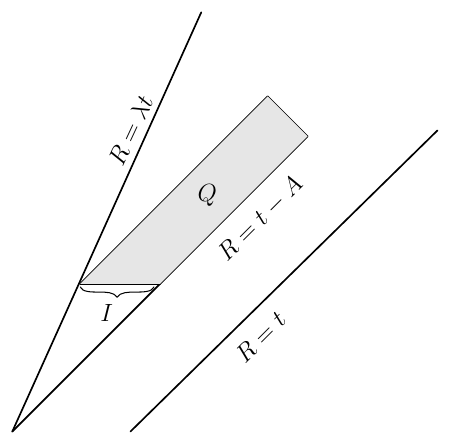}\end{center}%
\caption{The quadrilateral $Q$}\label{f:quad}
\end{figure}

By \eqref{e:algrel} the vector $(re, -rm)$ is divergence free, so we can conclude \begin{equation}\label{e:div}\begin{aligned} 0=&\int_{\partial Q} (re, -rm)\cdot \hat{n} \\
=& \underbrace{-\int_{\lambda T}^{T-A} erdr}_{I} - \underbrace{\int_{2T-A}^{2s-(1-\lambda)T} \mathcal A^2(A, \eta')d\eta'}_{II}\\
+& \underbrace{\int_A^{(1-\lambda)T} \mathcal B^2(\xi',2s- (1-\lambda)T)d\xi'}_{III} - \underbrace{\int_{(1+\lambda)T}^{2s-(1-\lambda)T} \mathcal A^2((1-\lambda)T, \eta')d\eta'}_{IV}.\end{aligned}\end{equation}
The terms $I-IV$ correspond to integrating along the sides of the quadrilateral $Q$ in Figure \ref{f:quad}, beginning at $I$ and moving counter-clockwise.

\medskip
In \eqref{e:div}, integral $I$ is exactly the one we want to show goes to zero as $T, A$ go to infinity, and hence we have
\[
|I| \leq |II| + |III| + |IV|.
\]
We will handle each term on the right separately.  Note that 
\[
II \leq \int_{T}^\infty \mathrm{Flux}(u, t, A)dt,
\]
and hence letting $T\uparrow \infty$, \eqref{e:fluxidentity} implies that $II$ goes to zero. 

For the next terms, we introduce some notation: let 
\[
\mathcal E_{\lambda}(\xi) = \int_{\frac{1+\lambda}{1-\lambda}\xi}^\infty \mathcal A^2(\xi, \eta')d\eta'.
\]
\begin{figure}
\begin{center}\includegraphics[width=.35\textwidth]{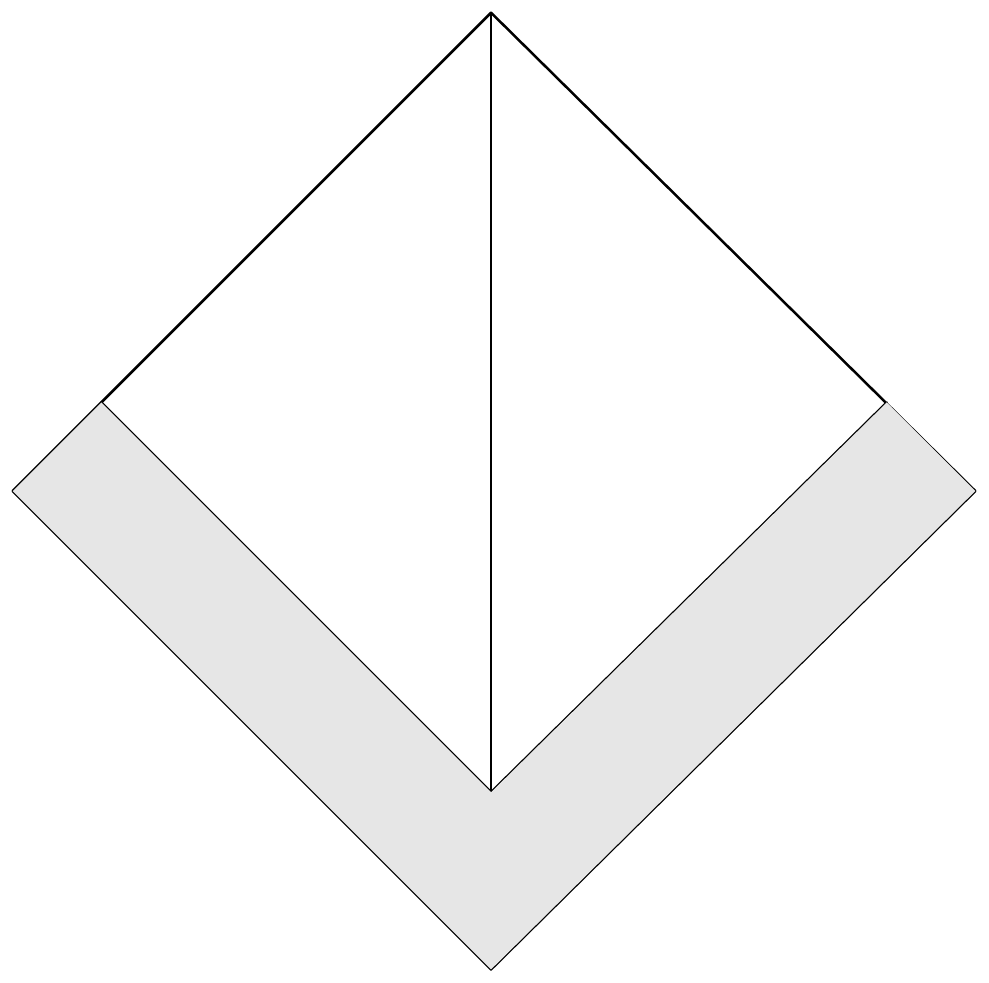}\end{center}%
\caption{The shaded region is created by taking a lightcone with vertex $(0, a)$, removing the lightcone with vertex $(0,b)$ and intersecting that region with the backwards lightcone with vertex $(0,\eta_0)$}\label{f:twolightcones}
\end{figure}

Now we handle term III. Integrating between two light cones, i.e. the shaded region in Figure \ref{f:twolightcones}, we see that 
\[
\int_{a}^b \mathcal B^2(\xi', \eta_0)d\xi' = \int_a^{\eta_0} \mathcal A^2(a, \eta')d\eta' - \int_{b}^{\eta_0} \mathcal A^2(b, \eta')d\eta'.
\]
Define 
 \[
 \mathcal F(a,b) =\lim_{\eta \rightarrow \infty}\int_a^b \mathcal B^2(\xi', \eta)d\xi',
 \]
and note that the limit in this definition exists since  
\begin{align}\label{fid}
\mathcal F(a,b) = \mathcal E_0(a) - \mathcal E_0(b).
\end{align}
The identity \eqref{fid} also implies that $\mathcal E_0(-)$ is a decreasing function since $\mathcal F(a,b) \geq 0$. Therefore, since $\mathcal F(a,b) \leq \mathcal E_0(a) < \mathcal{E}(u_0, u_1)$, we can define 
\[
\mathcal F(a) \equiv \lim_{b\rightarrow \infty} \mathcal F(a,b) .
\]
Now, noting that
\[
III \leq \mathcal F(A),
\]
we need to establish the limit of the right-hand side is zero. But since we know that $\mathcal E_0(-)$ is decreasing and non-negative, $\lim_{b\rightarrow \infty} \mathcal E_0(b)$ exists, and, consequently, $\lim_{a\rightarrow \infty} \mathcal F(a) =0$, which concludes the proof for term $III$.

 Finally, it is clear that $IV \leq \mathcal E_\lambda((1-\lambda)T)$, and hence we are done if we can show that 
 \begin{equation}\label{e:decayofenergy}
 \lim_{T\uparrow \infty} \mathcal E_{\lambda}((1-\lambda)T) = 0.
 \end{equation}
 Here we argue exactly as in \cite[Lemma 1]{CTZ93}. While that lemma is written in the setting of radially symmetric wave maps (with arbitrary target), we note that all the relevant quantities in our setting are radial by Proposition \ref{p:radialderivative}. Therefore, \eqref{e:decayofenergy} holds and the proof is complete. 
\end{proof}

We can then quickly conclude that the kinetic energy of a quasi-equivariant wave map vanishes inside of the light cone (this corresponds to Corollary 2.2 in \cite{ckls2}):

\begin{corollary}\label{c:kineticdispersion}
Let $u$ be a finite energy quasi-equivariant wave map with $T_{\max} = \infty$. Then \begin{equation}\label{e:kineticdispersion} \lim_{A \rightarrow \infty}\limsup_{T\rightarrow \infty} \frac{1}{T} \int_A^T \int_0^{t-A} |\partial_t u|^2(r,t)rdrdt = 0.\end{equation}
\end{corollary}

\begin{proof}
From \eqref{e:algrel}, it follows that 
\[
\partial_t(r^2m) = \partial_r(r^2e) - r|\partial_t u|^2_g.
\]
The proof then follows as in \cite[Corollary 2.2]{ckls2}.
\end{proof}

Now we narrow our focus to solutions to the system \eqref{equ:wm_sys}. We aim to prove the following: 

 \begin{proposition}\label{p:no_scatt}
For wave maps, $u$, given by \eqref{equ:wm_sys}, scattering cannot occur when $\alpha$ is not of degree zero. In particular, if $T_{\max}(u) = + \infty$ then there exists $t_n \uparrow \infty$ and $\lambda(t_n) \ll t_n$ such that $u(t_n + \lambda_nt, \lambda(t_n)r)\rightarrow \omega$ in $H^1_{\mathrm{loc}}$, where $\omega:\bR^{2} \rightarrow \mathcal N$ is a non-trivial harmonic map. 
\end{proposition}

 Our arguments will follow very closely those in \cite[Sections 2 and 3]{ckls2}, where an analogous proposition is proven for equivariant wave maps into spheres, see \cite[Theorem 3.2]{ckls2}. The $\alpha$ component of our wave map is not equivariant wave map itself (it satisfies a different equation), however, we will be able to use, essentially unchanged, any arguments in \cite{ckls2} which are purely energy theoretic, since $E(\alpha) \leq E(u)$.  

The first step is to show that the $\alpha$ component converges to a constant when $r \geq \lambda t$. Here we can argue exactly as in \cite{ckls2} with our Proposition \ref{p:channelsgotozero} and Corollary \ref{c:kineticdispersion} taking the place of \cite[Proposition 2.1 and Corollary 2.2]{ckls2}. 

\begin{corollary}[Cf. \protect{\cite[Corollary 2.3]{ckls2}}]\label{c:constantoutsideofcone}
Let $\lambda > 0$ and $u = (Y,\alpha)$ be a finite energy wave map which solves the system \eqref{equ:wm_sys}. Then $\alpha(\infty,t) := \lim_{r\rightarrow \infty} \alpha(r, t)$ exists and, furthermore, \[
\lim_{t\uparrow \infty} \|\alpha(r,t) - \alpha(\infty,t)\|_{L^\infty(r \geq \lambda t)} = 0.
\]
\end{corollary}

From Corollary \ref{c:constantoutsideofcone} we can construct our sequence $\lambda(t_n)$:

\begin{lemma}\label{l:constructionofscales}
Let $u = (Y,\alpha)$ be a finite energy wave map which solves \eqref{equ:wm_sys}. Further assume that $\alpha$ is not of zero degree. For each $t > T_0$, let $\lambda(t)$ be such that \begin{equation}\label{e:lambdaenergy} 1 \leq 2\pi\int_{0}^{2\lambda(t)} e(r,t) rdr \leq 2,\end{equation} then $\lambda(t) \ll t$. 
\end{lemma}

We first note that if $\alpha$ has degree not equal to zero then $E(u) > \mathcal E_{\mathbb S^2} = 2\pi$. So, by continuity of the energy, there exist $\lambda(t)$ which satisfy \eqref{e:lambdaenergy}. 

\begin{proof}
Imagine that there are $t_n \uparrow \infty$ and a $\lambda > 0$ such that $\lambda(t_n) > \lambda t_n$. It would then follow that $$\begin{aligned}\int_0^{2\lambda(t_n)} e(r,t_n)rdr \geq& \int_0^{\lambda t_n} e(r,t_n)rdr \geq \int_0^{\lambda t_n} \left((\partial_r \alpha)^2 + \frac{\sin^2(\alpha)}{2r^2}\right)rdr\\ \geq& \int_0^{\lambda t_n} \left|\partial_r \alpha\right|\left|\sin(\alpha)\right| dr \geq \int_0^{\alpha(\lambda t_n)} \sin(\rho) d\rho \stackrel{\mathrm{Corollary}\; \ref{c:constantoutsideofcone}}{\rightarrow} \int_0^{\pi} \sin(\rho)d\rho = 2.\end{aligned}$$ This contradicts the defintion of $\lambda(t)$ (i.e. \eqref{e:lambdaenergy}). As such $\lambda(t_n) \ll \lambda t_n$ for any $t_n \uparrow \infty$ and $\lambda > 0$. 
\end{proof}

We can now finish just as in the proof of \cite[Theorem 3.2]{ckls2}, we sketch the argument below:

\begin{proof}[Sketch of Proof of Proposition \ref{p:no_scatt}]
Let $\lambda(t)$ be defined as in Lemma \ref{l:constructionofscales} and let $A(t)$ be such that $A(t) \rightarrow \infty$ as $t\rightarrow \infty$ but also $\lambda(t) \leq A(t) \ll t$. Arguing from Corollary \ref{c:kineticdispersion} (see \cite[Lemma 3.3]{ckls2}) there exists $t_n \uparrow \infty$ and $\lambda_n := \lambda(t_n), A_n := A(t_n)$ such that $$\lim_{n\rightarrow \infty} \frac{1}{\lambda_n}\int_{t_n}^{t_n + \lambda_n} \int_0^{t-A_n}|\partial_t u|^2 rdrdt = 0.$$ Rescaling so that $u_n(x,t) = u(\lambda_n x, \lambda_n t +t_n)$ gives that \begin{equation}\label{e:timederivzero}\int_0^1 \int_0^{r_n}|\partial_t u_n|_g^2rdrdt \rightarrow 0,\end{equation} where $r_n = (t_n-A_n)/\lambda_n \rightarrow \infty$. 

This scaling also preserves the energy. Thus the $u_n \rightharpoonup u_\infty$ in $\dot{H}^1_{\mathrm{loc}}$. Furthermore, $u_\infty$ is time independent by \eqref{e:timederivzero} and therefore is a finite energy harmonic map, which can be assumed to be smooth and globally defined by the work of \cite{helein} and \cite{SacksUhlenbeck}. We now want to show that this convergence is strong, which also implies that $u_\infty$ is non-trivial by \eqref{e:lambdaenergy}. 

Away from $0$, we may assume the sequence converges locally uniformly by local uniform H\"older continuity of Lemma \ref{lem:holder}. Near $0$ we use \eqref{e:lambdaenergy} to see that $u_n(B_1(0))$ is contained in a coordinate chart in $\mathcal N$ (and the chart can be taken uniformly for $n$ large enough). Then one can argue exactly as in \cite[Theorem 2.1]{struweequivariant} (see in particular (3.17) onwards) to conclude that the convergence is strong near $0$. 
\end{proof}

\section{An Energy Gap for Bubbles arising in the flow}\label{s:harmonicmaps}

In this section we study the space of harmonic maps which can possible arise as bubbles in the flow \eqref{equ:wm_sys}. We begin with the following consequence of Lemma \ref{lem:holder}.

\begin{proposition}\label{p:C0_conv}
Let $u: \mathbb R^{1+2} \rightarrow \mathcal M$ be a quasi-equivariant wave-map to the smooth manifold $\mathcal M$. Let $t_n \uparrow T_{\max{u}}$ and $\lambda_n > 0$. If $u^{(n)}$ (defined in Theorem \ref{thm:st_dich}) converges in $H^1_{loc}$ to a non-trivial harmonic map, $\omega: \bS^2 \rightarrow \mathcal M$ then, for each $t$, up to passing to a susbsequence,
\[
u^{(n)}(t, x) \to \omega: \mathbb{R}^2 \to \mathcal{M},
\]
in $C^0_{loc}(\mathbb{R}^2 \setminus \{0\},\, \mathcal{M})$ as $n \to \infty$.
\end{proposition}

\begin{proof}
By Lemma \ref{lem:holder}, the sequence $u^{(n)} = u(t_n + \lambda_n t, \lambda_n x)$ is uniformly H\"older continuous on any compact region $K \subset \subset  \mathbb{R}^2 \setminus \{0\}$. By Arzel\'a-Ascoli, to extract a uniformly convergent subsequence, hence by a diagonal argument and passing to a further subsequence, we can product a subsequence which converges uniformly on any compact subset of $\mathbb{R}^2 \setminus \{0\}$. By uniqueness of limits, we have convergence to $\omega$. 
\end{proof}

\medskip
From this convergence, we can conclude first that the image of the bubble is topologically connected to the image of the rest of the flow.

\begin{proposition}[Connectedness of the flow]\label{p:connectedflow}
Let $u: \mathbb R^{1+2} \rightarrow \mathcal M$ be a quasi-equivariant wave map into a smooth manifold $\mathcal M$, such that bubbling (as above) occurs at time $T_{\max}(u)$ and that the rescaled wave maps converge to the non-trivial harmonic map $\omega: \mathbb S^2 \rightarrow \mathcal M$. Then 
\[
\bigcup_{t < T_{max}(u)} u(t, \mathbb R^2) \cup \omega(\mathbb S^2)
\]
 is a connected subset of $\mathcal M$. 
\end{proposition}

\begin{proof}
The fact that union of the images of the wave map components 
\[
\mathcal{U} :=  \bigcup_{t < T_{max}(u)} u(t, \mathbb R^2)
\]
and the image of the harmonic map $\omega(\mathbb S^2)$ are each connected follows from the continuity of the involved maps (in both space and time in the case of the wave maps). To see that the union of these two components is connected, we note that by Proposition \ref{p:C0_conv}, the set $\omega(\mathbb{R}^2 \setminus\{0\})$ contains limit points of $\mathcal{U}$, which implies the result.
\end{proof}

The following immediate corollary of Propositions \ref{p:C0_conv} and \ref{p:connectedflow} constrains the form and image of any bubble arising in the flow \eqref{equ:wm_sys}:

\begin{corollary}\label{cor:bubbling_sym}
Any harmonic map $\omega:\bS^2\rightarrow \cN$ (where $\cN$ is the manifold constructed in Section \ref{s:target}) arising as a non-trivial bubble in the flow \eqref{equ:wm_sys} must be of the form \begin{equation}\label{eqn:hm_ansatz}\omega(r,\theta) = (0, Y(r), \alpha(r), \theta),\end{equation} and satisfy $\omega(\bS^2) \subset \overline{\gamma}\times_f\bS^2$. 
\end{corollary}

We are now ready to state and prove the main result of this section: a classification of the bubbles arising in the flow \eqref{equ:wm_sys} at low energies. Before we do, let us recall from Lemma \ref{lem:hmspheres} that  $\mathcal{E}_{\mathbb S^2}$ is the smallest possible energy of a non-trivial harmonic map $\omega: \bS^2 \rightarrow \bS^2$.  Also recall, from above, that $P_1: \cN \rightarrow \bT^2$ is the projection map onto the first coordinate and $P_2: \cN \rightarrow \bS^2$ is the projection map onto the second coordinate. Finally, using stereographic projection (which is energy preserving) we can think of $\omega: \mathbb R^2 \rightarrow \bS^2$ and will do so when convenient.

\begin{lemma}\label{lem:bubbling_loc}
Let $\omega: \bS^2 \rightarrow \cN$ be a non-trivial harmonic map arising as a bubble in the flow \eqref{equ:wm_sys}. Let $\varepsilon_0 = \min\{ \bar{\varepsilon}, (\cot(7\pi/16)/2C)^2\} > 0$, where $\bar{\varepsilon} > 0$ is the constant given by Theorem \ref{t:hmregularity} and $C>0$ is the constant given by \eqref{e:interioroscillationestimate}. Then, either $\mathcal E(\omega) = \mathcal{E}_{\bS^2}$, or $\mathcal E(\omega) > \mathcal{E}_{\bS^2}+ \varepsilon_0/2$. In the former case, $P_1\omega$ is a constant map into the set $\{w = 0\}$ and $P_2\omega$ is a degree one harmonic map from $\bS^2 \rightarrow \bS^2$.
\end{lemma}

In this proof we will use $Y, \alpha$ interchangeably with $P_1\omega, P_2\omega$ respectively. 

\begin{proof}
By Lemma \ref{l:finitepaths}, $\mathrm{Im}\,P_1\omega$ is either a subset of $\gamma$ or of $\{w = 0\} \equiv \overline{\gamma}\,\setminus \,\gamma$. In the latter case, $f$ is constant on $\{w = 0\}$ so $P_2\omega$ is a harmonic map between spheres and $P_1\omega$ is a harmonic map from $\mathbb S^2 \rightarrow \mathbb S^1$. The maximum principle implies that $P_1\omega$ is a constant, so $P_2\omega$ must be non-trivial, as $\omega$ is non-trivial.  Therefore,  $P_2\omega$ is either a degree one harmonic map or $\mathcal E(P_2\omega) \geq 2 \mathcal E_{\bS^2}$, and the conclusion of the lemma holds in both cases. 

Thus, it suffices to rule out the presence of harmonic maps where $\mathrm{Im}\, P_1\omega \subset \gamma$ satisfying our energy constraint. We first claim that $P_2\omega$ must be non-trivial. Indeed, if $P_2\omega$ is trivial, then $P_1\omega$ is a bounded harmonic map from $\bR^2$ to $\gamma(-\infty, \infty) \cong \bR$. Such a map must be trivial which contradicts the non-triviality of $\omega$. Moreover, since $P_2\omega$ is non-trivial and equivariant, the maximum principle implies that the image of $P_2\omega$ cannot be contained in a hemisphere, and thus $\mathcal E(P_2\omega) \geq \mathcal E_{\bS^2}$. 

Let 
\[
S = \{r \in \bR^2\mid f(0, Y(r)) < 2\},
\]
and assume that $\mathcal E(\omega) < \mathcal E_{\bS^2} + \varepsilon_0/2$.  Since 
  \[
  \int_{\mathbb R^2\backslash S}f(0,Y) e(\alpha)rdr > 2\int_{\mathbb R^2\backslash S} e(\alpha)rdr,
  \] 
we can estimate,
\[
\mathcal E_{\bS^2} + \varepsilon_0/2 > \mathcal E(\omega) \geq  \int_{\mathbb R^2 \backslash S} f(0, Y)e(\alpha)rdr + \int_S e(\alpha)r\,dr \geq 2\mathcal E(P_2\omega) - \int_S e(\alpha)r\,dr,
\]
 and recalling that $\mathcal E(P_2\omega) \geq \mathcal E_{\bS^2}$, it must be the case that 
 \begin{equation}\label{e:largeonS}
 \int_{S} e(\alpha)rdr > \mathcal E_{\mathbb S^2} - \varepsilon_0/2.
 \end{equation}
Again recalling that $\mathcal E_{\bS^2} + \varepsilon_0/2 > \mathcal E(\omega)$, we get
\[
\int_{\bR^2\backslash S} |D\omega|_g^2 < \varepsilon_0.
\]
Thus the hypothesis of Theorem \ref{t:hmregularity} are fulfilled inside of any ball $B \subset \bR^2 \backslash S$. $P_1 \omega = Y$ is radial, so the oscillation of $Y$ inside any component of $\bR^2\backslash S$ is bounded by $2C\varepsilon^{1/2}_0$. Note that $M > 2$ so $f(0,Y(r)) < 2$ implies that $|Y| > \cot(7\pi/16)$ (see \eqref{e:defoff}, \eqref{e:chi}). As $2C\varepsilon^{1/2}_0 < \cot(7\pi/16)$, we conclude that $Y$ does not change sign on $\mathbb R^2$.

Composing with the diffeomorphism $\Psi_s: (x,y,\alpha,\theta) \mapsto (x, y+s, \alpha,\theta)$ we see that 
\begin{equation}\label{e:notharmonic}
\frac{d}{ds}\mathcal{E}(\Psi_s\circ \omega)|_{s=0} = \int_{\bR^2} \partial_yf(0,Y) e(\alpha)rdr \neq 0,
\end{equation}
where the last ``non-equality" follows from \eqref{e:largeonS} and the fact that $Y$ does not change sign on $\bR^2$ (and, consequently, $\partial_y f(0,Y)$ has a sign). That the deformation in \eqref{e:notharmonic} changes the energy to first order contradicts the fact that $\omega$ is harmonic. Thus there are no harmonic maps $\omega$ with $P_1 \omega \subset \gamma$ satisfying our energy constraint, and our proof is complete. 
\end{proof}

\begin{remark}
It is not hard to show that $\mathcal E_{\mathbb S^2}$ is exactly the energy of the lowest energy non-trivial quasi-equivariant harmonic map, $\omega:\mathbb S^2\rightarrow \mathcal N$, what we call $\mathcal E_{\mathrm{quasi}}(\mathcal N)$ in Theorem \ref{t:main}. 
\end{remark}

\section{Properties of the singularity}\label{s:winding}

We now turn to the proof of our main theorem, which we recall here:

\begin{ntheorem}[Main theorem]
There exists a compact smooth Riemannian manifold $(\mathcal{N}, g)$ given by
\[
\mathcal{N} = \mathbb{T}^2 \times_f \mathcal{S}^2
\]
for a certain $C^\infty(\bT^2)$ warping function, $f$, and $C^\infty$-smooth, finite energy, quasi-equivariant initial data $(u_0, u_1)$, which satisfy
\[
\|(u_0, u_1) \|_{\dot H^3 \times \dot H^2} < \infty, \qquad \mathcal E(u_0, u_1) < \mathcal E_{\mathrm{quasi}}(\mathcal N) + \varepsilon_1,
\]
such that the corresponding solution $(u, u_t)$ to \eqref{eq:wm} has a bubbling singularity as $t \to T_{\max}$ which fails to be unique in the sense of Definition \ref{d:solitonresolution}. Above, $\varepsilon_1 >0$ is a constant which depends only on $\mathcal N$, and $\mathcal E_{\mathrm{quasi}}(\mathcal N)$ denotes the smallest energy of a non-trivial quasi-equivariant harmonic map, $\omega: \mathbb S^2 \rightarrow \mathcal N$.
\end{ntheorem}

Before proceeding with the proof of the main result we summarize what we know so far about solutions to  \eqref{eq:wm}, or more precisely to solutions of \eqref{equ:wm_sys}. In light of Sterbenz and Tataru's dichotomy, Theorem \ref{thm:st_dich}, at every finite time $t_0$, the solution either concentrates energy at a point and bubbles off (a Lorentz transform of) a finite energy harmonic map or it can be continued smoothly past $t_0$. If the solution exists as $t\rightarrow \infty$ then it either scatters in the limit or bubbles off (a Lorentz transform of) a finite energy harmonic map.  If the ``spherical part" of the initial data has degree one, Proposition \ref{p:no_scatt} implies that blow-up in the form of bubbling off a harmonic map must occur (either in finite time or as $t\rightarrow \infty$).

Now note that quasi-equivariance implies that there is no Lorentz transform symmetry in the solutions, and hence the convergence of the bubbling sequence (as described in part (a) of Theorem \ref{thm:st_dich}) must be to an entire, non-constant harmonic map. Moreover, singularity formation at any point in the domain requires the bubbling of at least as much energy as the lowest energy non-trivial harmonic map. Hence by additivity of the energy, bubbling can occur only at finitely many points, and so in the quasi-equivariant setting the only possible blow-up point in the domain is the origin $x = 0$. 

Finally, if the initial conditions \eqref{e:initconditions} have energy less than $\mathcal E_{\mathbb S^2} + \varepsilon_0/2$, then it must be the case that the bubble also has energy less than this. By Lemma \ref{lem:bubbling_loc}, this implies that the harmonic map, $\omega$, decomposes into a constant map into $\{w=0\}\subset \bT^2$ and a degree one equivariant harmonic map between spheres. 

\medskip
Summarizing the previous discussion, we have, thus far, established that:

\begin{theorem}\label{thm:global_blow_up}
Consider the wave map equation \eqref{eq:wm} with smooth, quasi-equivariant symmetric, finite energy initial data of the form \eqref{e:initconditions}. Further assume that $\alpha_0$ (the ``spherical part" of the initial conditions) is a degree one map between 2-spheres.  Then there exists a unique solution to \eqref{eq:wm} on the time interval $[0, T_{\max})$, and a sequence of times $t_n\uparrow T_{\max}$ and scales $\lambda_n$ with
\begin{equation}
\lambda_n = \begin{cases} o(T_{max} - t_n) & T_{max} < \infty \\
 o(t_n) & T_{max} = \infty \end{cases}
\end{equation}
so that the rescaled sequence of maps
\[
u^{(n)}(t, x) = u(\lambda_n t+ t_n, \lambda_n x)
\]
converges strongly in $H^1_{loc}$ to an entire Harmonic map 
\[
\omega : \mathbb{R}^2 \to \mathcal{N}
\]
of nontrivial energy. Furthermore, if  $$\mathcal E(u_0, u_1) < \mathcal E_{\mathbb S^2} + \varepsilon_0/2,$$ then $\omega = (P_1\omega, P_2\omega)$ satisfies that $P_1\omega = (0, z_0)$ for some $z_0 \in [0,1]$ and $P_2\omega$ is a degree one equivariant harmonic maps between 2-spheres.
\end{theorem}

We note that it is easy to produce initial conditions which satisfy the hypothesis of Theorem~\ref{thm:global_blow_up}. Using the notation of \eqref{e:initconditions}, one such example is letting $(\alpha_0(r), \theta)$ be a degree one equivariant harmonic map, $\alpha_1(r) \equiv 0 \equiv Y_1(r)$ and $Y_0(r) \equiv c$ where $c$ is a constant large enough such that $f(0,c) < 1 + \varepsilon_0/(2\mathcal E_{\mathbb S^2})$. It is important to observe that even though $(\alpha_0, \theta)$ is a harmonic map into $\mathbb S^2$ and $Y_0$ is a harmonic map into $\mathbb T^2$ the map $u_0(r, \theta) := (0, Y_0, \alpha_0(r), \theta): \mathbb R^{1+2} \rightarrow \mathbb T^2 \times_f \mathbb S^2$ is \emph{not} a harmonic map (as one reduces energy to first order by ``sliding down" the geodesic $\gamma$). As such these initial conditions do not lead to a stationary wave map (see Proposition \ref{p:winding} below). 

\subsection{Winding Singularities}\label{ss:winding}

In this subsection we prove that the bubbling guaranteed by Theorem \ref{thm:global_blow_up} is winding in the sense of \cite{toppingwinding}, and in fact enjoys a stronger form of winding which we introduce below, which we call ``strongly winding''. We will then prove that a strongly winding bubble implies non-uniqueness in the sense of Definition \ref{d:solitonresolution}. Our stronger notion of winding requires winding along \emph{all} sequences, as opposed to just one. We find this definition to be slightly easier to work with, and, in the situation where only one bubble develops at a given point and time, equivalent to the standard definition of winding. 

\medskip
 Before we introduce the following stronger notion of winding recall that if $\mathcal{M}$ is a Riemannian manifold, then $\widehat{\mathcal M}$ its its universal cover and, given a function $u : \mathbb{R}^{1+d} \to \mathcal{M}$, we let $\hat{u}: \mathbb R^{1+d} \rightarrow \widehat{\mathcal M}$ be its lift to the universal cover. 

\begin{definition}[Strongly winding bubbling]\label{d:strong_winding}
A quasi-equivariant wave map $u: \mathbb R^{1+2} \rightarrow \mathcal M$ exhibits strongly winding bubbling at time $T_{\max}$ and the origin $ 0 \in \mathbb R^2$ if for any sequences $\{r_n\}$, and $\{t_n\}$, satisfying
\begin{equation}
 t_n \uparrow T_{\max}, 
 \qquad r_n = \begin{cases} o(T_{max} - t_n) & T_{max} < \infty \\
 o(t_n) & T_{max} = \infty \end{cases}
\end{equation}
such that $u(t_n, r_nx) \rightarrow \omega(x)$ in $C^0_{\mathrm{loc}}(\bR^2\backslash \{0\}; \mathcal M)$, where $\omega$ is a non-constant harmonic map, the lifts $\hat{u}(t_n,r_nx)$ have no convergent subsequence in $C^0_{\mathrm{loc}}(\bR^2\backslash \{0\}; \widehat{\mathcal M})$. 
\end{definition}

For a solution $u$ with winding bubbling, we say that $u$ has {\it a winding bubble with respect to sequences $\{t_n\}$ and $\{r_n\}$}, when $\{t_n\}$ and $\{r_n\}$ are the sequences guaranteed by Definition \ref{d:winding}. The following proposition provides an equivalent definition of (strongly) winding.

\begin{proposition}\label{r:unboundedwinding}
If $u:\mathbb R^{1+d} \rightarrow \mathcal M$ develops a bubble at $(T_{\max}, 0)$, then the bubble is winding with respect to sequences $\{t_n\}$ and $\{r_n\}$ if and only if there exists a compact $K \subset \subset \bR^d \backslash \{0\}$ such that every subsequence 
\[
\{\hat{u}(t_{n_j}, r_{n_j}x)\big|_K\}
\]
is unbounded.
\end{proposition}

\begin{proof}
This proposition follows readily from the equicontinuity of the lifts and compactness, but we include the details for completeness. Suppose first that the bubble is winding in the sense of Definition \ref{d:winding}. Since the sequence $\{u^{(n)}\}$ is equicontinuous on every compact set, the family given by the lifts $\hat{u}^{(n)}:= \hat{u}(t_n, r_nx)$ is still equicontinuous on every compact set. Indeed, fix a compact set  $K \subset \subset \bR^d \backslash \{0\}$, $x \in K$ and $\varepsilon > 0$, and let $\delta > 0$ be as in the definition of equicontinuity of the family $\{u^{(n)}\}$. Making $\delta$ smaller if necessary, we can restrict to a sufficiently small neighborhood so that the covering map trivializes. This yields the equicontinuity of the lifts. Hence, the only way the lifted sequence can fail to be precompact is if there exists a compact set on which the family is unbounded.  

Conversely, if there exists such a compact subset $K$, then unboundedness implies that there is no convergent subsequence, and hence the singularity is winding. 
\end{proof}

\begin{remark}
The previous proposition adapts readily to strongly winding bubbling.
\end{remark}

We are now prepared to establish the following main proposition about the nature of the singularity.

\begin{proposition}\label{p:winding}
The flow described \eqref{equ:wm_sys} with initial conditions as in Theorem \ref{thm:global_blow_up} exhibits strongly winding bubbling at $(T_{max}, 0)$.
\end{proposition}

\begin{proof}
We know the flow needs to converge to a bubble as described in Lemma \ref{lem:bubbling_loc}. Let $\gamma$ be the geodesic described in Section \ref{s:target}. By the path lifting property, $\gamma$ may be lifted to a unique path $\hat{\gamma}$ in the universal cover, which is necessarily an unbounded curve since $\gamma$ is not null-homotopic in $\cN$ and $\pi_1(\cN) = \bZ^2$. 

Let now $\{t_n\}$ and $\{\lambda_n\}$ be any sequences along which
\[
u^{(n)}(t, x) = u(\lambda_nt + t_n, \lambda_n x)
\]
converge to $\omega$, see Theorem \ref{thm:global_blow_up}. It simplifies things to fix a time slice, so we abuse notation by letting $u^{(n)}(x) = u^{(n)}(0, x)$. Lemma \ref{lem:bubbling_loc} tells us that $P_1\omega \subset \{w=0\}$ which implies that for any $t_0 \gg 1$ and any compact set $K\subset \subset \mathbb R^2\backslash \{0\}$, there exists an $N = N(t_0, K)$ such that if $n \geq N$, then  $P_1u^{n}(K) \subset \gamma((-\infty, -t_0] \cup [t_0, \infty))$\footnote{Recall that by Proposition \ref{p:C0_conv}, $P_1 u^{n}\rightarrow P_1\omega$ in $C^0(K)$ and that $\{w= 0\}$ is exactly the set of limit points of $\gamma(t)$ as $t\rightarrow \pm \infty$}. Lifting to the universal cover, the unboundedness of $\hat{\gamma}$ implies that the sequence $\hat{u}^{(n)}|_{K}:= \hat{u}(t_n, r_nx)|_{K}$ must be unbounded, and hence by Proposition \ref{r:unboundedwinding}, the singularity is winding. 
\end{proof}


Our final result shows that a strongly winding bubbling cannot give rise to a unique bubble in the sense of Definition \ref{d:solitonresolution}.

\begin{theorem}\label{t:nonuniqueness}
If $u$ exhibits strongly winding bubbling at $x_0$ and time $T_{\max}$, then it does not have a unique bubble at that point and time.
\end{theorem}

\begin{proof}
Suppose for contradiction that  $u$ exhibits strongly winding bubbling at the origin at time $T_{max}$, and a unique bubble $\omega$ to which it converges under rescaling by $r(t)$.
and translation by $x(t)$. 

We will show that for every compact  $K \subset \subset \bR^d \backslash \{0\}$, there exists a subsequence 
\[
\{\hat{u}(t_{n_j}, r_{n_j}x)\big|_K\}
\]
which is bounded, contradicting Proposition \ref{r:unboundedwinding}. Fix a compact set  $K \subset \subset \bR^d \backslash \{0\}$ and note that $\widehat{\omega}(K)$ is compact. Define the $\varepsilon$-neighbourhood of a set $A$ by 
\[
N_\varepsilon(A) := \{x\mid \mathrm{dist}(x,A) < \varepsilon\}.
\]
Fix $\varepsilon  = 1$, since $u(t, r(t)- ) \rightarrow \omega$ in $C^0_{\mathrm{loc}}$, there exists $t_0$, such that for all $t > t_0$ we have 
\[
u(t, r(t)K) \subset N_1(\omega(K)).
\]
But we then have that $\{\hat{u}(t_n, r(t_n)K)\}_{\{n\mid t_n > t_0\}} \subset N_1(\widehat{\omega}(K))$, and, in particular, the sequence is bounded, which contradicts Proposition~\ref{r:unboundedwinding}, and concludes the proof.
\end{proof}

\bibliography{Winding-refs}{}

\providecommand{\bysame}{\leavevmode\hbox to3em{\hrulefill}\thinspace}
\providecommand{\MR}{\relax\ifhmode\unskip\space\fi MR }
\providecommand{\MRhref}[2]{%
  \href{http://www.ams.org/mathscinet-getitem?mr=#1}{#2}
}
\providecommand{\href}[2]{#2}
\begin{thebibliography}{DKMM22}

\bibitem[Bre05]{brendleyamabe}
Simon Brendle, \emph{Convergence of the {Y}amabe flow for arbitrary initial
  energy}, J. Differential Geom. \textbf{69} (2005), no.~2, 217--278.
  \MR{2168505}

\bibitem[CKL18]{CKL18}
Elisabetta Chiodaroli, Joachim Krieger, and Jonas L\"{u}hrmann,
  \emph{Concentration compactness for critical radial wave maps}, Ann. PDE
  \textbf{4} (2018), no.~1, Paper No. 8, 148. \MR{3749762}

\bibitem[CKLS15a]{ckls1}
R.~C{\^o}te, C.~E. Kenig, A.~Lawrie, and W.~Schlag, \emph{Characterization of
  large energy solutions of the equivariant wave map problem: {I}}, Amer. J.
  Math. \textbf{137} (2015), no.~1, 139--207. \MR{3318089}

\bibitem[CKLS15b]{ckls2}
R.~C{\^o}te, C.~E. Kenig, A.~Lawrie, and W.~Schlag, \emph{Characterization of
  large energy solutions of the equivariant wave map problem: {II}}, Amer. J.
  Math. \textbf{137} (2015), no.~1, 209--250. \MR{3318090}

\bibitem[CM15]{comi}
Tobias~Holck Colding and William~P. Minicozzi, II, \emph{Uniqueness of blowups
  and \l ojasiewicz inequalities}, Ann. of Math. (2) \textbf{182} (2015),
  no.~1, 221--285. \MR{3374960}

\bibitem[C{\^o}t15]{cote}
R.~C{\^o}te, \emph{On the soliton resolution for equivariant wave maps to the
  sphere}, Comm. Pure Appl. Math. \textbf{68} (2015), no.~11, 1946--2004.
  \MR{3403756}

\bibitem[CTZ93]{CTZ93}
Demetrios Christodoulou and A.~Shadi Tahvildar-Zadeh, \emph{On the regularity
  of spherically symmetric wave maps}, Comm. Pure Appl. Math. \textbf{46}
  (1993), no.~7, 1041--1091. \MR{1223662}

\bibitem[DdPW20]{delPinoWei}
Juan D\'{a}vila, Manuel del Pino, and Juncheng Wei, \emph{Singularity formation
  for the two-dimensional harmonic map flow into {$S^2$}}, Invent. Math.
  \textbf{219} (2020), no.~2, 345--466. \MR{4054257}

\bibitem[DJKM17]{djkm}
Thomas Duyckaerts, Hao Jia, Carlos Kenig, and Frank Merle, \emph{Soliton
  resolution along a sequence of times for the focusing energy critical wave
  equation}, Geom. Funct. Anal. \textbf{27} (2017), no.~4, 798--862.
  \MR{3678502}

\bibitem[DJKM18]{djkm2}
Thomas Duyckaerts, Hao Jia, Carlos Kenig, and Frank Merle, \emph{Universality
  of blow up profile for small blow up solutions to the energy critical wave
  map equation}, Int. Math. Res. Not. IMRN (2018), no.~22, 6961--7025.
  \MR{3878592}

\bibitem[DKM13]{dkm3d}
Thomas Duyckaerts, Carlos Kenig, and Frank Merle, \emph{Classification of
  radial solutions of the focusing, energy-critical wave equation}, Camb. J.
  Math. \textbf{1} (2013), no.~1, 75--144. \MR{3272053}

\bibitem[DKM19]{dkmfullsoliton}
Thomas Duyckaerts, Carlos~E. Kenig, and Frank Merle, \emph{Soliton resolution
  for the radial critical wave equation in all odd space dimensions}, 2019.

\bibitem[DKMM22]{DKMM}
Thomas Duyckaerts, Carlos Kenig, Yvan Martel, and Frank Merle, \emph{Soliton
  resolution for critical co-rotational wave maps and radial cubic wave
  equation}, Comm. Math. Phys. \textbf{391} (2022), no.~2, 779--871.
  \MR{4397184}

\bibitem[dP17]{delpino}
Manuel del Pino, \emph{Bubbling blow-up in critical parabolic problems},
  Nonlocal and nonlinear diffusions and interactions: new methods and
  directions, Lecture Notes in Math., vol. 2186, Springer, Cham, 2017,
  pp.~73--116. \MR{3588122}

\bibitem[FM19]{Feehan1}
Paul M.~N. Feehan and Manousos Maridakis, \emph{Lojasiewicz-simon gradient
  inequalities for analytic and morse-bott functions on banach spaces}, to
  appear in Crelle's Jounral, 2019.

\bibitem[Gri17]{Grinis}
Roland Grinis, \emph{Quantization of time-like energy for wave maps into
  spheres}, Comm. Math. Phys. \textbf{352} (2017), no.~2, 641--702.
  \MR{3627409}

\bibitem[H{\'e}l91]{helein}
Fr\'{e}d\'{e}ric H{\'e}lein, \emph{R\'{e}gularit\'{e} des applications
  faiblement harmoniques entre une surface et une vari\'{e}t\'{e}
  riemannienne}, C. R. Acad. Sci. Paris S\'{e}r. I Math. \textbf{312} (1991),
  no.~8, 591--596. \MR{1101039}

\bibitem[Jen17]{jendrej17}
Jacek Jendrej, \emph{Construction of type {II} blow-up solutions for the
  energy-critical wave equation in dimension 5}, J. Funct. Anal. \textbf{272}
  (2017), no.~3, 866--917. \MR{3579128}

\bibitem[Jen19]{jendrejconstruction}
Jacek Jendrej, \emph{Construction of two-bubble solutions for energy-critical
  wave equations}, Amer. J. Math. \textbf{141} (2019), no.~1, 55--118.
  \MR{3904767}

\bibitem[JK17]{jiakenig}
Hao Jia and Carlos Kenig, \emph{Asymptotic decomposition for semilinear wave
  and equivariant wave map equations}, Amer. J. Math. \textbf{139} (2017),
  no.~6, 1521--1603. \MR{3730929}

\bibitem[JL18]{jendrejlawrie}
Jacek Jendrej and Andrew Lawrie, \emph{Two-bubble dynamics for threshold
  solutions to the wave maps equation}, Invent. Math. \textbf{213} (2018),
  no.~3, 1249--1325. \MR{3842064}

\bibitem[JL21]{JL2}
Jacek Jendrej and Andrew Lawrie, \emph{Soliton resolution for energy-critical
  wave maps in the equivariant case}, 2021.

\bibitem[Kri04]{kriegerhyperbolic}
Joachim Krieger, \emph{Global regularity of wave maps from {$\bold R^{2+1}$} to
  {$H^2$}. {S}mall energy}, Comm. Math. Phys. \textbf{250} (2004), no.~3,
  507--580. \MR{2094472}

\bibitem[KS97]{KlS97}
Sergiu Klainerman and Sigmund Selberg, \emph{Remark on the optimal regularity
  for equations of wave maps type}, Comm. Partial Differential Equations
  \textbf{22} (1997), no.~5-6, 901--918. \MR{1452172}

\bibitem[KS12]{kriegerschlagbook}
Joachim Krieger and Wilhelm Schlag, \emph{Concentration compactness for
  critical wave maps}, EMS Monographs in Mathematics, European Mathematical
  Society (EMS), Z\"{u}rich, 2012. \MR{2895939}

\bibitem[KST08]{KST08}
J.~Krieger, W.~Schlag, and D.~Tataru, \emph{Renormalization and blow up for
  charge one equivariant critical wave maps}, Invent. Math. \textbf{171}
  (2008), no.~3, 543--615. \MR{2372807}

\bibitem[LO16]{lawire_oh_cmp}
Andrew Lawrie and Sung-Jin Oh, \emph{A refined threshold theorem for
  {$(1+2)$}-dimensional wave maps into surfaces}, Comm. Math. Phys.
  \textbf{342} (2016), no.~3, 989--999. \MR{3465437}

\bibitem[{\L}oj65]{lojasiewicz}
S.~{\L}ojasiewicz, \emph{Ensembles semi-analytiques}, IHES Notes, 1965.

\bibitem[Nah13]{Nahas}
J.~Nahas, \emph{Scattering of wave maps from {$\Bbb R^{2+1}$} to general
  targets}, Calc. Var. Partial Differential Equations \textbf{46} (2013),
  no.~1-2, 427--437. \MR{3016515}

\bibitem[Pil19]{Pillai}
Mohandas Pillai, \emph{Infinite time blow-up solutions to the energy critical
  wave maps equation}, arXiv:1905.00167 (2019).

\bibitem[QT97]{qingtian}
Jie Qing and Gang Tian, \emph{Bubbling of the heat flows for harmonic maps from
  surfaces}, Comm. Pure Appl. Math. \textbf{50} (1997), no.~4, 295--310.
  \MR{1438148}

\bibitem[Riv95]{riviere}
Tristan Rivi\`ere, \emph{Everywhere discontinuous harmonic maps into spheres},
  Acta Math. \textbf{175} (1995), no.~2, 197--226. \MR{1368247}

\bibitem[RR12]{RR12}
Pierre Rapha\"{e}l and Igor Rodnianski, \emph{Stable blow up dynamics for the
  critical co-rotational wave maps and equivariant {Y}ang-{M}ills problems},
  Publ. Math. Inst. Hautes \'{E}tudes Sci. \textbf{115} (2012), 1--122.
  \MR{2929728}

\bibitem[RS10]{rodster}
Igor Rodnianski and Jacob Sterbenz, \emph{On the formation of singularities in
  the critical {${\rm O}(3)$} {$\sigma$}-model}, Ann. of Math. (2) \textbf{172}
  (2010), no.~1, 187--242. \MR{2680419}

\bibitem[Rup21]{rupflin}
Melanie Rupflin, \emph{Lojasiewicz inequalities for almost harmonic maps near
  simple bubble trees}, 2021.

\bibitem[Sim83]{simonLoj}
Leon Simon, \emph{Asymptotics for a class of nonlinear evolution equations,
  with applications to geometric problems}, Ann. of Math. (2) \textbf{118}
  (1983), no.~3, 525--571. \MR{727703}

\bibitem[SS98]{shatahstruwe}
Jalal Shatah and Michael Struwe, \emph{Geometric wave equations}, Courant
  Lecture Notes in Mathematics, vol.~2, New York University, Courant Institute
  of Mathematical Sciences, New York; American Mathematical Society,
  Providence, RI, 1998. \MR{1674843}

\bibitem[ST10a]{sterbenztataruwavemaps2}
Jacob Sterbenz and Daniel Tataru, \emph{Energy dispersed large data wave maps
  in {$2+1$} dimensions}, Comm. Math. Phys. \textbf{298} (2010), no.~1,
  139--230. \MR{2657817}

\bibitem[ST10b]{sterbenztataruwavemaps1}
Jacob Sterbenz and Daniel Tataru, \emph{Regularity of wave-maps in dimension
  {$2+1$}}, Comm. Math. Phys. \textbf{298} (2010), no.~1, 231--264.
  \MR{2657818}

\bibitem[Str85]{Struwe85}
Michael Struwe, \emph{On the evolution of harmonic mappings of {R}iemannian
  surfaces}, Comment. Math. Helv. \textbf{60} (1985), no.~4, 558--581.
  \MR{826871}

\bibitem[Str88]{struweflow}
Michael Struwe, \emph{On the evolution of harmonic maps in higher dimensions},
  J. Differential Geom. \textbf{28} (1988), no.~3, 485--502. \MR{965226}

\bibitem[Str03a]{struweequivariant}
Michael Struwe, \emph{Equivariant wave maps in two space dimensions}, Comm.
  Pure Appl. Math. \textbf{56} (2003), no.~7, 815--823, Dedicated to the memory
  of J\"{u}rgen K. Moser. \MR{1990477}

\bibitem[Str03b]{Struweradial}
Michael Struwe, \emph{Radially symmetric wave maps from {$(1+2)$}-dimensional
  {M}inkowski space to general targets}, Calc. Var. Partial Differential
  Equations \textbf{16} (2003), no.~4, 431--437. \MR{1971037}

\bibitem[Str13]{StruweLecNotes}
Michael Struwe, \emph{Wave maps with and without symmetries}, Evolution
  equations, Clay Math. Proc., vol.~17, Amer. Math. Soc., Providence, RI, 2013,
  pp.~483--510. \MR{3098646}

\bibitem[SU81]{SacksUhlenbeck}
J.~Sacks and K.~Uhlenbeck, \emph{The existence of minimal immersions of
  {$2$}-spheres}, Ann. of Math. (2) \textbf{113} (1981), no.~1, 1--24.
  \MR{604040}

\bibitem[SU82]{schoenuhlenbeck}
Richard Schoen and Karen Uhlenbeck, \emph{A regularity theory for harmonic
  maps}, J. Differential Geometry \textbf{17} (1982), no.~2, 307--335.
  \MR{664498}

\bibitem[SW20]{waldron}
Chong Song and Alex Waldron, \emph{Harmonic map flow for almost-holomorphic
  maps}, 2020.

\bibitem[Tao]{taoarxiv}
Terence Tao, \emph{Global regularity of wave maps {III}-{VII}}, Preprints
  2008-2009.

\bibitem[Tao01]{taosmalldata}
Terence Tao, \emph{Global regularity of wave maps. {II}. {S}mall energy in two
  dimensions}, Comm. Math. Phys. \textbf{224} (2001), no.~2, 443--544.
  \MR{1869874}

\bibitem[Tat05]{tatarurough}
Daniel Tataru, \emph{Rough solutions for the wave maps equation}, Amer. J.
  Math. \textbf{127} (2005), no.~2, 293--377. \MR{2130618}

\bibitem[Top97]{toppingrigidity}
Peter~Miles Topping, \emph{Rigidity in the harmonic map heat flow}, J.
  Differential Geom. \textbf{45} (1997), no.~3, 593--610. \MR{1472890}

\bibitem[Top00]{toppingthesis}
Peter Topping, \emph{An example of a nontrivial bubble tree in the harmonic map
  heat flow}, Harmonic morphisms, harmonic maps, and related topics ({B}rest,
  1997), Chapman \& Hall/CRC Res. Notes Math., vol. 413, Chapman \& Hall/CRC,
  Boca Raton, FL, 2000, pp.~185--191. \MR{1735698}

\bibitem[Top04a]{toppingannals}
Peter Topping, \emph{Repulsion and quantization in almost-harmonic maps, and
  asymptotics of the harmonic map flow}, Ann. of Math. (2) \textbf{159} (2004),
  no.~2, 465--534. \MR{2081434}

\bibitem[Top04b]{toppingwinding}
Peter Topping, \emph{Winding behaviour of finite-time singularities of the
  harmonic map heat flow}, Math. Z. \textbf{247} (2004), no.~2, 279--302.
  \MR{2064053}

\bibitem[Whi92]{whitenonunique}
Brian White, \emph{Nonunique tangent maps at isolated singularities of harmonic
  maps}, Bull. Amer. Math. Soc. (N.S.) \textbf{26} (1992), no.~1, 125--129.
  \MR{1108901}

\end{thebibliography}
\bibliographystyle{amsbeta}

\end{document}